\documentclass[11pt,letterpaper]{article}

\pdfoutput=1

\usepackage{graphics,epsfig,graphicx,color}
\usepackage[caption = false]{subfig}
\usepackage{float}
\usepackage{capt-of}

\usepackage{amssymb,latexsym}

\usepackage{setspace}

\usepackage{amsmath}

\usepackage{amsthm,amsxtra}

\usepackage[final]{showkeys}

\usepackage{stmaryrd}

\usepackage{algorithm}
\usepackage{algorithmicx}
\usepackage{algpseudocode}

\usepackage[openbib]{currvita}
\usepackage[square,numbers,sort]{natbib}

\usepackage{etaremune}
\usepackage{enumitem} 

\newif\ifPDF
\ifx\pdfoutput\undefined
	\PDFfalse
\else
	\ifnum\pdfoutput > 0
		\PDFtrue
	\else
		\PDFfalse
	\fi
\fi

\ifPDF
	\usepackage{pdftricks}
	\begin{psinputs}
		\usepackage{pstricks}
		\usepackage{pstcol}
		\usepackage{pst-plot}
		\usepackage{pst-tree}
		\usepackage{pst-eps}
		\usepackage{multido}
		\usepackage{pst-node}
		\usepackage{pst-eps}
	\end{psinputs}
\else
	\usepackage{pstricks}
\fi

\ifPDF
	\usepackage[debug,pdftex,colorlinks=true, 
	linkcolor=blue, bookmarksopen=false,
	plainpages=false,pdfpagelabels]{hyperref}
\else
	\usepackage[dvips]{hyperref}
\fi

\usepackage{fancyhdr}

\usepackage{comment}

\usepackage[toc,page]{appendix}

\usepackage{cancel}

\usepackage{multirow}
\usepackage{tabularx}

\newtheorem{theorem}{Theorem}[section]

\newtheorem{remark}[theorem]{Remark}

\newtheorem{assumption}{Assumption}
\newtheorem{truth}{Theorem}

\newcommand{\bbR}{\mathbb R}

 \newcommand{\bx}{\mathbf x}

\DeclareMathOperator{\sech}{sech}

\newenvironment{keywords}
{\noindent{\bf Key words.}\small}{\par\vspace{1ex}}
\newenvironment{AMS}
{\noindent{\bf AMS subject classifications 2020.}\small}{\par}

\newcommand{\DELETE}[1]{}

\newlength{\bibitemsep}
\newlength{\bibparskip}
\let\oldthebibliography\thebibliography
\renewcommand\thebibliography[1]{%
  \oldthebibliography{#1}%
  \setlength{\parskip}{\bibitemsep}%
  \setlength{\itemsep}{\bibparskip}%
}

\title{An energy-based discontinuous Galerkin method for the nonlinear Schr\"{o}dinger equation with wave operator}

\author{
Kui Ren \thanks{Department of Applied Physics and Applied Mathematics, Columbia University, New York, NY 10027, USA. Email: kr2002@columbia.edu} 
\and 
Lu Zhang 
\thanks{Department of Computational Applied Mathmatics and Operations Research, Rice University, Houston, TX 77005, USA. Email: lz82@rice.edu} 
\and  
Yin Zhou \thanks{Department of Applied Physics and Applied Mathematics, Columbia University, New York, NY 10027, USA. Email: yz3888@columbia.edu}
}

\begin{document}

\maketitle

\begin{abstract}
This work develops an energy-based discontinuous Galerkin (EDG) method for the nonlinear Schr\"{o}dinger equation with the wave operator. The focus of the study is on the energy-conserving or energy-dissipating behavior of the method with some simple mesh-independent numerical fluxes we designed. We establish error estimates in the energy norm that require careful selection of a test function for the auxiliary equation involving the time derivative of the displacement variable. A critical part of the convergence analysis is to establish the $L^2$ error bounds for the time derivative of the approximation error in the displacement variable by using the equation that determines its mean value. Using a specially chosen test function, we show that one can create a linear system for the time evolution of the unknowns even when dealing with nonlinear properties in the original problem. Extensive numerical experiments are provided to demonstrate the optimal convergence of the scheme in the $L^2$ norm with our choices of the numerical flux.
\end{abstract}
 




\section{Introduction}

Nonlinear Schr\"{o}dinger equations with wave operators are widely used in physical applications~\cite{BaDoXi-PhysicaD10,Xin-PhysicaD00}. The equations are natural generalizations of linear Schr\"{o}dinger equations to include different nonlinear effects. In this work, we consider a specific version of the nonlinear Schr\"{o}dinger equation with the wave operator, given in the following form:
\begin{equation}\label{EQ:NLSW}
\begin{aligned}
   & \frac{\partial^2 u}{\partial t^2} - \Delta u + i \alpha \frac{\partial u}{\partial t} + \beta({\bf x}) f(|u|^2) u = 0, \quad {\bf x} \in \Omega, \ 0<t<T,\\
   & u({\bf x}, 0) = u_0({\bf x}), \quad \frac{\partial u}{\partial t}({\bf x}, 0) = u_1({\bf x}), \ \quad {\bf x} \in \Omega,
\end{aligned}
\end{equation}
with periodic boundary conditions in space. Here $\Omega$ is a bounded domain in $\mathbb{R}^d$ ($d$ being the dimension of the space), $\frac{\partial }{\partial t}$ and $\frac{\partial^2}{\partial t^2}$ denote the first-order and second-order time derivatives, respectively. $\Delta = \sum_{k = 1}^d \frac{\partial^2}{\partial x_k^2}$ is the Laplace operator, and $i$ is the standard imaginary unit (i.e. $i^2 = -1$). The wave function, $u(\bx, t)$, a complex function of space and time, is the main quantity of interest. The nonlinear effect of the medium is described by the real-valued function $f:\bbR_0\to\bbR$. The strength of the nonlinearity is controlled by the parameter $\beta$. The parameter $\alpha$ is a real constant that describes the damping or amplification effect of the underlying medium. In terms of notation, we use the standard $|u|: = \sqrt{u\bar{u}}$ to denote the modulus of the complex number $u$ (where $\bar{u}$ represents the complex conjugate of $u$).

Model~\eqref{EQ:NLSW} has been extensively studied due to its relevance in various applications. We refer interested readers to~\cite{BaDoXi-PhysicaD10,BeCo-PhysicaD95,Matsuuchi-JPSJ80,scho1979klein,shuji2002klein,Xin-PhysicaD00} and references therein for some examples of such applications. Developing accurate and efficient numerical methods for the computational study of the model has been an active area of research in recent years. Finite difference schemes are among the most popular ones developed~\cite{DeLi-JCAM23,Wang-JMAA15}. The work of~\cite{GuLi-JNMCA83} was among the earliest in this direction. It presented an implicit nonconservative finite-difference scheme for~\eqref{EQ:NLSW}. Conditionally stable explicit conservative finite-difference methods are developed in~\cite{zhang2003findiff}. These schemes, however, need to be combined with another scheme to start the computation. The scheme presented in~\cite{zhang2003findiff} was later improved to be unconditionally stable in~\cite{wang2006findiff}. All these methods achieve second-order accuracy in space, with $L^2$ or energy error estimates, and are restricted to problems with one spatial dimension (i.e. $d = 1$). When the nonlinear term is cubic, fourth-order three-level compact finite difference schemes have been proposed in~\cite{LiZhWa-AMC12,LaOm-ANM22}, again in the one-dimensional setting. Other spatial discretization schemes have also been developed. In terms of spectral methods, ~\cite{Wang-JCM07} proposed a second-order multisymplectic Fourier pseudospectral method, ~\cite{BaCa-SIAM14,BaCaZh-SIAM14} developed pseudospectral methods with a multiscale time integrator, and~\cite{LiGoZh-JSC21} coupled the trigonometric pseudo-spectral method with high-order time difference scheme. In terms of the finite element method, ~\cite{li2015ldg} proposed a local discontinuous Galerkin (LDG) method for the problem. The scheme in~\cite{li2015ldg} is of arbitrary order, and the authors were able to prove optimal $L^2$ error estimates for simplified linear problems.

Special attention has been paid to several challenging issues in the numerical solution of~\eqref{EQ:NLSW}. These issues include, for instance, a) how to handle complicated boundary conditions for the model~\cite{LiWuZh-PRE14,Li-M2AS21,PaXiHeZh-ANM20}, for example, absorbing boundary conditions for simulating wave propagation in unbounded domains; b) how to construct conservative schemes such that the energy-conservation property of the physical system~\cite{YaLiGu-AMC21,ShZh-CMA22} can be ensured on the discrete level; and c) how to develop accurate time integrators to ensure numerical stability and control error propagation in long time integration of the system~\cite{CoFa-DCDS98,FeGuYu-arXiv23,GuCaJiWa-arXiv22,LiGoZh-JSC21}. Preserving asymptotic behavior of the solution in various scaling limits is also practically crucial and has been extensively studied~\cite{BaCa-SIAM14,BaZh-JCP19}.


The objective of this work is to develop an arbitrary order energy-based discontinuous Galerkin (EDG) method for the nonlinear Schr{\"o}dinger model~\eqref{EQ:NLSW} in any spatial dimension $d\geq 1$. We establish the optimal convergence rate of the scheme in an energy norm with a general nonlinear coupling $f$.

The energy-based discontinuous Galerkin method was first introduced in~\cite{appelo2015new} as a generalization of the standard discontinuous Galerkin (DG) method~\cite{CoKaSh-Book12,Riviere-Book08} to allow the direct handling of scalar wave equations in second-order forms. The EDG method is based on two fundamental principles: first, the introduction of a variable that is weakly equivalent to the time derivative of the solution, and second, the development of mesh-independent numerical fluxes based on the energy flux at the element boundaries. Since then, the EDG method has been successfully applied to various wave equations beyond its original formulation for second-order scalar wave equations. For example, \cite{appelo2018energy} extended the method to elastic wave equations by incorporating additional symmetries of the potential energy, \cite{zhang2019energy} applied it to the advective wave equation, where the energy does not depend solely on the sum of kinetic and potential energy, and \cite{hagstrom2021discontinuous} used the method to solve Maxwell’s equations in general linear
dispersive media where the problem is formulated in second-order form. The method has also been generalized to semilinear and nonlinear wave equations by designing a special weak formulation for the displacement equation, resulting in a linear system for its time derivative~\cite{appelo2020energy,zhang2023nonlinear}. Recent studies have focused on improving the performance of the method, such as improving the convergence rates ~\cite{du2019convergence}, dealing with numerical stiffness results from the polynomial approximation ~\cite{appelo2021stagger,zhang2021energy}, and extending the method to fourth-order semilinear wave equations ~\cite{zhang2021local}. Compared to other DG methods, EDG methods allow for simple conservative or upwind fluxes independent of the mesh, and only one auxiliary variable is introduced regardless of the dimension of the problem, making it an attractive option for solving a wide range of wave equations.

In this work, we extend the scope of the EDG method by applying it to the nonlinear Schr\"{o}dinger equation with the wave operator ~\eqref{EQ:NLSW}. To solve this problem, we introduce the time derivative of the solution variable and transform the equation into a first-order system in time. Notably, this system contains only two unknowns, regardless of the dimension of the problem. The proposed semi-discrete EDG method exhibits unconditional stability and does not require a penalty term. In addition, the scheme conserves or dissipates energy based on the choice of simple, mesh-independent numerical fluxes. To obtain optimal error estimates of the proposed method in an energy norm, we develop a special test function for the auxiliary equation of $\frac{\partial u}{\partial t}$ in the nonlinear Schr\"{o}dinger equation with the wave operator (see ~\eqref{eq:21}). Note that this test function facilitates both the analysis of optimal error estimates in an energy norm and the computation of the time derivative of unknowns for the nonlinear problem ~\eqref{EQ:NLSW}. To control the $L^2$ error of the displacement caused by the nonlinear coupling in the problem ~\eqref{EQ:NLSW}, we introduce an additional positive nonlinear volume integral in the corresponding error energy equation, as described in ~\eqref{eq:err_energy}. Finally, we use the mean value equation of the time derivative of the error in the displacement variable to constrain its $L^2$ norm, which ultimately leads to optimal error estimates in the energy norm on structured grids.

The rest of the paper is organized as follows. In Section~\ref{sec_formula}, we present the reformulation of the problem~\eqref{EQ:NLSW} and the EDG semi-discretization, along with different interelement fluxes, and establish a basic energy estimate. We then derive the error estimates in the energy norm for the proposed EDG method in Section~\ref{sec:error}. Numerical experiments in one and two dimensions are provided in Section~\ref{sec:numerical} to verify our theoretical results. Concluding remarks are offered in Section~\ref{sec:conclusion}.

\section{Energy-based DG scheme}\label{sec_formula}


To obtain an EDG formulation for problem ~\eqref{EQ:NLSW}, we introduce a new variable $v = \frac{\partial u}{\partial t}$, which leads to the first-order system in time:
\begin{equation}\label{system}
\left\{
\begin{aligned}
u_t - v &= 0, \\
v_t - \Delta u + i \alpha v + \beta({\bf x}) f(|u|^2) u &= 0.
\end{aligned}
\right.
\end{equation}
Here and in the following the subindex $t$ represents the partial derivative with $u_t = \frac{\partial u}{\partial t}$. We introduce the following energy functional for the system ~\eqref{system} 
\begin{equation}
\label{eq:energy}
E(t) := \int_{\Omega} |v|^2 + |\nabla u|^2 + \beta({\bf x}) F(|u|^2) d {\bf x},
\end{equation}
where $F(s) = \int_0^s f(r) dr$. Under periodic boundary conditions, we find that the energy is conserved over time as
\[\frac{dE(t)}{dt} = \int_\Omega 2\mbox{Re}(v\bar{v}_t) + 2\mbox{Re}(\nabla u\cdot \nabla \bar{u}_t) + \beta({\bf x}) f(|u|^2)\frac{d|u|^2}{dt}\ d{\bf x}= \int_{\partial \Omega} 2\mbox{Re}(\nabla u \cdot {\bf n} \bar{u})\ dS = 0,\]
where ${\bf n}$ is the outward unit normal, and $\mbox{Re}(w)$ is the real part of the complex-valued function $w$. That is,
\begin{equation*}
E(t) = E(0) \quad \forall\ 0<t<T.
\end{equation*}
In the next subsection, we introduce the EDG semi-discretization of system ~\eqref{system} and several interelement fluxes and establish a fundamental energy estimate.

\subsection{Semi-discrete EDG formulation}
\label{sec:semi_discrete}

Let us now introduce our EDG approach for ~\eqref{EQ:NLSW} based on the reformulation ~\eqref{system}. The domain $\Omega$ is partitioned into non-overlapping elements $\Omega_j$, where $\Omega = \cup_j\Omega_j$. To approximate the solution components $(u,v)$ on each element $\Omega_j$, we use either complex polynomials or the tensor product of complex polynomials of degree $q$ and $s$, respectively. Specifically, we define the sets $U_h^q$ and $V_h^s$ as
\begin{align*}
U_h^q &= \{u^h({\bf x},t)\big| u^h({\bf x},t) \in \mathcal{P}^q(\Omega_j), {\bf x} \in \Omega_j, 0 \le t \le T\}, \\
V_h^s &= \{v^h({\bf x},t)\big| v^h({\bf x},t) \in \mathcal{P}^s(\Omega_j), {\bf x} \in \Omega_j, 0 \le t \le T\}.
\end{align*}
Next, we look for an approximation to the system ~\eqref{system} that satisfies a discrete energy estimate, analogous to ~\eqref{eq:energy}. We consider the following discrete energy for each $\Omega_j$
\begin{equation}\label{dis_energy}
E_j^h(t) := \int_{\Omega_j} |v^h|^2 + |\nabla u^h|^2\ d{\bf x} + \sum_{{\bf k}}\omega_{j,{\bf k}}\beta({\bf x}_{j,{\bf k}})F(|u^h({\bf x}_{j,{\bf k}}, t)|^2),
\end{equation}
where we use a quadrature rule with nodes ${\bf x}_{j, {\bf k}}$ in $\Omega_j$ and positive weights $\omega_{j, {\bf k}}$ to approximate the integration of the nonlinear term involving $F(|u^h|)$ in ~\eqref{eq:energy}. Then, the time derivative of the discrete energy is
\begin{equation}\label{dt_dis_energy}
    \frac{dE_j^h(t)}{dt} = \int_{\Omega_j} 2\mbox{Re}(v^h\bar{v}^h_t) + 2\mbox{Re}(\nabla u^h\cdot \nabla \bar{u}^h_t) \ d{\bf x} + \sum_{\bf k} \omega_{j,{\bf k}}\beta({\bf x}_{j, {\bf k}}) f(|u^h({\bf x}_{j,{\bf k}},t)|^2)\frac{d|u^h({\bf x}_{j,{\bf k}},t)|^2}{dt}.
\end{equation}
To ensure the accuracy of the quadrature rule for integrating nonlinear terms, we impose the following condition.
\begin{assumption}\label{assump}
The quadrature rule satisfies $\forall \Omega_j$,
    \begin{align}
    \sum_{\bf k} \omega_{j,{\bf k}} \phi^2({\bf x}_{j,{\bf k}}) - \int_{\Omega_j} \phi^2\ d{\bf x} &= 0,  \nonumber\\
    \sum_j\Big| \sum_{\bf k} \omega_{j, {\bf k}}\phi({\bf x}_{j, {\bf k}})g({\bf x}_{j, {\bf k}}) - \int_{\Omega_j} \phi g \ d{\bf x} \Big| &\leq Ch^{q+1}\|\phi\|_{L^2(\Omega_h)}|g|_{H^{q+1}(\Omega_h)},\nonumber
    \end{align}
    for all $\phi \in \mathcal{P}^q(\Omega_j)$ and  $g \in H^{q+1}(\Omega_h)$. Here, $C$ is a constant independent of $h$. 
\end{assumption}
To ensure that the weak form is compatible with the discrete energy equations ~\eqref{dis_energy} and ~\eqref{dt_dis_energy}, we choose $\phi_u \in U_h^q$ and $\phi_v \in V_h^s$ to test the first and second equations in the system ~\eqref{system} with $-\Delta \phi$ and $\psi$, respectively. We also add terms that vanish in the continuous problem. This results in the following equations,
\begin{align*}
    &\int_{\Omega_j} - \Delta \phi (u^h_t - v^h)\ d{\bf x} + \sum_{\bf k}\omega_{j,{\bf k}}\beta({\bf x}_{j,{\bf k}}) f(|u^h({\bf x}_{j,{\bf k}})|^2) \phi({\bf x}_{j,{\bf k}}) \big(u^h_t({\bf x}_{j,{\bf k}}) - v^h({\bf x}_{j,{\bf k}})\big)  \\
    = & \int_{\partial \Omega_j} {\bf n} \cdot \nabla \phi (v^\ast - u^h_t)\ dS, \\
    &\int_{\Omega_j} v^h_t \psi - \Delta u^h \psi + i \alpha v^h \psi\ d{\bf x} + \sum_{\bf k}\omega_{j,{\bf k}}\beta({\bf x}_{j,{\bf k}}) f(|u^h({\bf x}_{j, {\bf k}})|^2) u^h({\bf x}_{j, {\bf k}}) \psi({\bf x}_{j, {\bf k}}) \\
    =& \int_{\partial \Omega_j} {\bf n} \cdot \big((\nabla u)^\ast  - \nabla u^h \big)\psi\ dS.
\end{align*}
Here we have omitted $t$ in $u({\bf x}_{j,{\bf k}})$ for simplicity, and introduced numerical fluxes $v^\ast$ and $(\nabla u)^\ast$ for interelement boundaries. An integration by parts leads us to the alternative forms,
\begin{align}
    \label{eq:21}
    &\int_{\Omega_j} \nabla \phi \cdot \nabla (u^h_t - v^h)\ d{\bf x} + \sum_{\bf k}\omega_{j,{\bf k}}\beta({\bf x}_{j,{\bf k}}) f(|u^h({\bf x}_{j,{\bf k}})|^2) \phi({\bf x}_{j,{\bf k}}) \big(u^h_t({\bf x}_{j,{\bf k}}) - v^h({\bf x}_{j,{\bf k}})\big) \nonumber \\
    = & \int_{\partial \Omega_j} {\bf n} \cdot \nabla \phi (v^\ast - v^h) \ dS,\\
    \label{eq:22}
    &\int_{\Omega_j} v^h_t \psi + \nabla u^h \cdot \nabla \psi + i \alpha v^h \psi \ d{\bf x} + \sum_{\bf k}\omega_{j,{\bf k}}\beta({\bf x}_{j,{\bf k}}) f(|u^h({\bf x}_{j, {\bf k}})|^2) u^h({\bf x}_{j, {\bf k}}) \psi({\bf x}_{j, {\bf k}}) \nonumber \\
    = &\int_{\partial \Omega_j} {\bf n} \cdot (\nabla u)^\ast \psi \ dS.
\end{align}
Choosing $\phi = \bar{u}^h$ and $\psi = \bar{v}^h$ in ~\eqref{eq:21} and ~\eqref{eq:22}, respectively, we collect

\begin{align}
    \label{eq:21_1}
   & \int_{\Omega_j} \nabla \bar{u}^h \cdot \nabla (u^h_t - v^h) \ d{\bf x}+ \sum_{{\bf k}}\omega_{j,{\bf k}}\beta({\bf x}_{j,{\bf k}}) f(|u^h({\bf x}_{j,{\bf k}})|^2) \bar{u}^h({\bf x}_{j,{\bf k}}) \big(u^h_t({\bf x}_{j,{\bf k}}) - v^h({\bf x}_{j,{\bf k}})\big) \nonumber \\
   =& \int_{\partial \Omega_j} {\bf n} \cdot \nabla \bar{u}^h (v^* - v^h) \ dS,\\
    \label{eq:22_1}
   & \int_{\Omega_j} v^h_t \bar{v}^h + \nabla u^h \cdot \nabla \bar{v}^h + i \alpha v^h \bar{v}^h \ d{\bf x}+ \sum_{{\bf k}}\omega_{j,{\bf k}}\beta({\bf x}_{j,{\bf k}}) f(|u^h({\bf x}_{j,{\bf k}})|^2) u^h({\bf x}_{j,{\bf k}}) \bar{v}^h({\bf x}_{j,{\bf k}}) \nonumber\\
   = &\int_{\partial \Omega_j} {\bf n} \cdot (\nabla u)^* \bar{v}^h \ dS.
\end{align}
Then, we compute the complex conjugate of  ~\eqref{eq:21_1}--\eqref{eq:22_1}, add them, and use equation ~\eqref{dt_dis_energy} to get
\begin{equation*}
    \frac{d E^h}{dt} = \sum_j \frac{d E^h_j}{dt} = \sum_j \int_{\partial \Omega_j} {\bf n} \cdot 2 \mbox{Re} \Big(\nabla \overline{u}^h (v^\ast - v^h) + (\nabla u)^\ast \overline{v}^h\Big)\ dS.
\end{equation*}
To ensure that the EDG scheme ~\eqref{eq:21_1}--\eqref{eq:22_1} is uniquely solvable for any $\beta({\bf x})f(|u|^2)$, an additional equation is included to determine the mean value of $\frac{\partial u^h}{\partial t}$. This equation is given by
\begin{equation}\label{eq:mean_value}
\int_{\Omega_j} \tilde{\phi}(u^h_t - v^h)\ d{\bf x} = 0 \quad \forall \tilde{\phi} \in \mathcal{P}^0(\Omega_j).
\end{equation}
Note that the equation ~\eqref{eq:mean_value} does not affect the energy term and is also sufficient to ensure that equation ~\eqref{eq:21_1} holds when $\phi \in \mathcal{P}^0(\Omega_j)$. Moreover, it is consistent with the first equation in the original system ~\eqref{system}.

\subsection{Numerical fluxes}

It is necessary to define the numerical fluxes for interelement boundaries to establish the EDG scheme proposed in Section \ref{sec:semi_discrete}. We label two elements sharing an interelement boundary face, denoted as $F_j$, by $1$ and $2$. We introduce standard notations for averages and jumps given by 
\[{{v^h}} = \frac{1}{2}(v_1^{h} + v_2^{h}), \quad [[v^h]] = v_1^{h}{\bf n}_1 + v_2^{h}{\bf n}_2,\]
and
\[{{\nabla u^h}} = \frac{1}{2}(\nabla u_1^{h} + \nabla u_2^{h}), \quad [[\nabla u^h]] = \nabla u_1^{h} \cdot {\bf n}_1 + \nabla u_2^{h} \cdot {\bf n}_2.\]
Then, the net contribution to the discrete energy $E^h(t)$ from each element $\Omega_j$ is given by $\int_{F_j} 2J^h\ dS$ with
\begin{equation}
\label{eq:inn}
J^h = {\bf n}_1 \cdot \mbox{Re}\Big(\nabla \bar{u}_1^h (v^\ast - v_1^h) + (\nabla u)^\ast \bar{v}_1^h\Big) + {\bf n}_2 \cdot \mbox{Re}\Big(\nabla \bar{u}_2^h (v^\ast - v_2^h) + (\nabla u)^\ast \bar{v}_2^h\Big).
\end{equation}
To create an energy-stable scheme, we select numerical fluxes such that $J^h \leq 0$. If $J^h < 0$, the method is classified as dissipative, while $J^h = 0$ indicates a conservative method. To achieve this goal, we introduce the following numerical fluxes,
\begin{equation}
    \label{eq:generalfl}
    v^* = \mu v^h_1 + (1-\mu) v^h_2 - \tau [[\nabla u^h]], \quad (\nabla u)^* = (1-\mu) \nabla u^h_1 + \mu \nabla u^h_2 - \gamma [[v^h]],
\end{equation}
where $\mu \in \mathbb{R}$ and $\gamma, \tau \ge 0$. Plugging ~\eqref{eq:generalfl} into ~\eqref{eq:inn}, we have
\begin{equation*}
    J^h = - \mbox{Re} \Big(\gamma \big|[[v^h]]\big|^2 + \tau \big|[[\nabla u^h]]\big|^2\Big) \le 0.
\end{equation*}
Note that the commonly used \emph{central fluxes} can be obtained by setting $\mu = 1/2$ when $\gamma = \tau = 0$. This can be expressed as 
\begin{equation}\label{flux:central}
v^\ast = \{\{v^h\}\},\quad (\nabla u)^\ast = \{\{\nabla u^h\}\}.
\end{equation}
 The resulting scheme conserves energy with $J^h =0$. On the other hand, if $\gamma = \tau = 0$ and $\mu$ is equal to $0$ or $1$, \emph{alternating fluxes} are obtained, namely, 
\begin{equation}\label{flux:a1}
v^\ast = v_1^h, \quad (\nabla u)^\ast = \nabla u_2^h,
\end{equation}
or
\begin{equation}\label{fluxes:a2}
v^\ast = v_2^h, \quad (\nabla u)^\ast = \nabla u_1^h.
\end{equation}
This also leads to an energy conserving scheme with $J^h = 0$. The \emph{Sommerfeld fluxes} are obtained with $\mu = 1/2$, $\gamma = \xi/2$, with $\tau = 1/(2\xi)$, where $\xi > 0$, i.e, 
\begin{equation}\label{flux:up}
v^\ast = \{\{v^h\}\} - \frac{1}{2\xi} [[\nabla u^h]], \quad (\nabla u)^\ast = \{\{\nabla u\}\} - \frac{\xi}{2}[[v^h]].
\end{equation}
This scheme is energy dissipating because it yields $J^h < 0$.

Now, we are ready to establish the stability of the proposed EDG scheme.
\begin{truth}\label{thm1}
The discrete energy $E^h(t) = \sum_j E^h_j(t)$, where $E^h_j(t)$ defined in ~\eqref{dis_energy}, satisfies
\begin{equation*}
    \frac{d E^h}{dt} = \sum_j \frac{dE_j^h}{dt} = - \sum_j \int_{F_j} 2 \mbox{Re} \Big(\gamma \big|[[v^h]]\big|^2 + \tau \big|[[\nabla u^h]]\big|^2\Big) \ dS,
\end{equation*}
where $F_j$ is an interelement boundary. If the flux parameters $\tau$ and $\gamma$ are non-negative, then the discrete energy is non-increasing such that
\begin{equation*}
    E^h(t) \le E^h(0)\quad  \forall\  0<t<T.
\end{equation*}

\end{truth}

\section{Error estimates}\label{sec:error}

To study the numerical error of the proposed EDG scheme, we introduce the errors as follows,
\begin{equation*}
    e_u = u - u^h, \quad e_v = v - v^h
\end{equation*}
and compare the EDG solution $(u^h, v^h)$ with an arbitrary polynomial $(\underline{u}^h, \underline{v}^h)$, where $(\underline{u}^h, \underline{v}^h) \in (U^q_h,V^s_h)$ with 
\begin{equation}\label{eq:degree}
q-2 \leq s \leq q.
\end{equation}
To continue the analysis, we introduce the differences
\begin{equation*}
    \underline{e}_u := \underline{u}^h - u^h, \quad \underline{e}_v := \underline{v}^h - v^h, \quad \delta_u := \underline{u}^h - u, \quad \delta_v := \underline{v}^h - v
\end{equation*}
and define the numerical error energy as
\begin{equation}\label{eq:err_energy}
    \mathcal{E} 
    := \sum_j \int_{\Omega_j} |\underline{e}_v|^2 + |\nabla \underline{e}_u|^2 \ d{\bf x}+ \sum_{j, {\bf k}}\int_0^{|\underline{e}_u({\bf x}_{j, {\bf k}})|^2} \omega_{j,{\bf k}}\beta({\bf x}_{j,{\bf k}}) f(|\underline{u}^h({\bf x}_{j, {\bf k}})|^2+z) dz.
\end{equation}
Here, for the positivity of the error energy $\mathcal{E}$, we assume that there is a positive constant $L$ such that 
\begin{equation}\label{eq:lower}
    \beta({\bf x}) f(|u|^2) \ge L. 
\end{equation}
We note that this assumption can be relaxed as discussed in Remark \ref{remark1}. The term $\int_0^{|\underline{e}_u({\bf x}_{j,{\bf k}})|^2}$ $ \omega_{j,{\bf k}}\beta({\bf x}_{j,{\bf k}}) f(|\underline{u}^h({\bf x}_{j,{\bf k}})|^2+z) dz$ is introduced to allow a bound on $|\underline{e}_u|^2$ in the error analysis. Specifically, we have
\begin{equation*}
    \int_0^{|\underline{e}_u({\bf x}_{j, {\bf k}})|^2}\omega_{j, {\bf k}} \beta({\bf x}_{j, {\bf k}}) f(|\underline{u}^h({\bf x}_{j, {\bf k}})|^2 + z) dz \ge L\omega_{j, {\bf k}} |\underline{e}_u({\bf x}_{j,{\bf k}})|^2.
\end{equation*}

On the other hand, since both the continuous solution $(u, v)$ and the EDG solution $(u^h, v^h)$ satisfy the equations ~\eqref{eq:21}--\eqref{eq:22}, we have the following equations 
\begin{equation}\label{eq:err1}
\begin{aligned}
    \int_{\Omega_j} \nabla \phi \cdot &\nabla \big((e_u)_t - e_v\big)\ d{\bf x} + \int_{\Omega_j} \beta({\bf x}) f(|u|^2) \phi (u_t  - v) \ d{\bf x} \\ 
    & - \sum_{\bf k}\omega_{j, {\bf k}}\beta({\bf x}_{j,{\bf k}}) f(|u^h({\bf x}_{j,{\bf k}})|^2) \phi({\bf x}_{j,{\bf k}}) (u^h_t({\bf x}_{j,{\bf k}})  - v^h\big({\bf x}_{j,{\bf k}})\big) = \int_{\partial \Omega_j} {\bf n} \cdot \nabla \phi \big(e_v^\ast - e_v\big)\ dS,
\end{aligned}
\end{equation}
and
\begin{equation}\label{eq:err2}
\begin{aligned}
\int_{\Omega_j} (e_v)_t \psi &+ \nabla e_u \cdot \nabla \psi + i \alpha e_v \psi\ d{\bf x} + \int_{\Omega_j} \beta({\bf x}) f(|u|^2)u\psi\ d{\bf x}\\
&- \sum_{\bf k} \omega_{j,{\bf k}} \beta({\bf x}_{j, {\bf k}}) f(|u^h({\bf x}_{j, {\bf k}})|^2) u^h({\bf x}_{j, {\bf k}}) \psi({\bf x}_{j, {\bf k}}) = \int_{\partial \Omega_j} {\bf n} \cdot (\nabla e_u)^* \psi\ dS,
\end{aligned}
\end{equation}
where integrals are used to replace discrete quadrature when the continuous solution $(u,v)$ is in place. Now, by choosing
\[\phi = \underline{e}_{\bar{u}} := \underline{\bar{u}}^h - \bar{u}^h, \quad \psi = \underline{e}_{\bar{v}} := \underline{\bar{v}}^h - \bar{v}^h\] 
in ~\eqref{eq:err1} and ~\eqref{eq:err2}, respectively, and invoking the relations
\begin{equation}\label{eq:error_relation}
e_u = \underline{e}_u - \delta_u, \quad e_v = \underline{e}_v - \delta_v, 
\end{equation}
we can then take the complex conjugate of the equations ~\eqref{eq:err1} and ~\eqref{eq:err2} and add the resulting equations with ~\eqref{eq:err1} and ~\eqref{eq:err2} to get 
\begin{align}
    & \int_{\Omega_j} 2\mbox{Re}\big(\nabla \underline{e}_{\bar{u}} \cdot \nabla ( \underline{e}_u )_t \big) + 2\mbox{Re}\big( ( \underline{e}_{v} )_t \underline{e}_{\bar{v}} \big) \ d{\bf x} 
    = \int_{\Omega_j} 2\mbox{Re}\Big(\nabla \underline{e}_{\overline{u}} \cdot \nabla \big((\delta_u)_t + \underline{e}_v - \delta_v\big) \Big) \ d{\bf x}\nonumber\\
    +& \!\int_{\Omega_j} \!\! 2\mbox{Re}\big( ( \delta_v )_t \underline{e}_{\bar{v}} \big)+2\mbox{Re}\big( \nabla (-\underline{e}_u + \delta_u) \cdot \nabla \underline{e}_{\bar{v}} \big) - 2\mbox{Im}\big( \alpha \delta_v \underline{e}_{\overline{v}} \big) 
    + 2\mbox{Re}\big( \beta({\bf x}) f(|u|^2) \underline{e}_{\bar{u}} (- u_t  + v) \big) \ d{\bf x} \nonumber\\
   - & \underbrace{\sum_{j,{\bf k}} 2\mbox{Re}\Big( \omega_{j,{\bf k}}\beta({\bf x}_{j,{\bf k}}) f(|u^h({\bf x}_{j, {\bf k}})|^2) \underline{e}_{\bar{u}}({\bf x}_{j,{\bf k}}) \big(-u^h_t({\bf x}_{j,{\bf k}})  + v^h({\bf x}_{j,{\bf k}}) \big) \Big)}_{\mathcal{N}_1}  \label{eq_full}\\
    + & \underbrace{\sum_{j,{\bf k}} 2\mbox{Re} \Big( \omega_{j,{\bf k}}\beta({\bf x}_{j,{\bf k}}) \big(- f(|u({\bf x}_{j,{\bf k}})|^2) u({\bf x}_{j,{\bf k}}) + f(|u^h({\bf x}_{j,{\bf k}})|^2) u^h({\bf x}_{j,{\bf k}})\big) \underline{e}_{\overline{v}}({\bf x}_{j,{\bf k}}) \Big)}_{\mathcal{N}_2}  \nonumber \\
    +& \int_{\partial \Omega_j} 2\mbox{Re}\big( {\bf n} \cdot \nabla \underline{e}_{\bar{u}} (\underline{e}_v^\ast - \delta_v^\ast - \underline{e}_v + \delta_v) \big) + 2\mbox{Re}\big( {\bf n} \cdot (\nabla (\underline{e}_u - \delta_u))^\ast \underline{e}_{\bar{v}} \big) \ dS,\nonumber
\end{align}
where $\mbox{Im}(w)$ is the imaginary part of the complex-valued function $w$. Note that we can further integrate by parts on the volume integral $\int_{\Omega_j}2\mbox{Re} \big( \nabla \underline{e}_{\overline{u}} \cdot \nabla \delta_v \big) d{\bf x}$ to collect
\begin{align}
    & \int_{\Omega_j} 2\mbox{Re}\big(\nabla \underline{e}_{\bar{u}} \cdot \nabla ( \underline{e}_u )_t \big) + 2\mbox{Re}\big( ( \underline{e}_{v} )_t \underline{e}_{\bar{v}} \big) \ d{\bf x} 
    = \int_{\Omega_j} 2\mbox{Re}\big(\nabla \underline{e}_{\bar{u}} \cdot \nabla ((\delta_u)_t + \underline{e}_v) \big) + 2\mbox{Re}\big( \Delta \underline{e}_{\bar{u}} \delta_v \big)\ d{\bf x} \nonumber\\
    +& \!\int_{\Omega_j} \!\! 2\mbox{Re}\big( ( \delta_v )_t \underline{e}_{\bar{v}} \big) + 2\mbox{Re}\big( \nabla (-\underline{e}_u + \delta_u) \cdot \nabla \underline{e}_{\bar{v}} \big) - 2\mbox{Im}\big( \alpha \delta_v \underline{e}_{\overline{v}} \big) 
    +2\mbox{Re}\big( \beta({\bf x}) f(|u|^2) \underline{e}_{\bar{u}} (- u_t  + v) \big) \ d{\bf x} \nonumber\\
    - & \mathcal{N}_1 + \mathcal{N}_2 
    +\int_{\partial \Omega_j} 2\mbox{Re}\big( {\bf n} \cdot \nabla \underline{e}_{\bar{u}} (\underline{e}_v^\ast - \delta_v^\ast - \underline{e}_v ) \big) + 2\mbox{Re}\big( {\bf n} \cdot (\nabla (\underline{e}_u - \delta_u))^\ast \underline{e}_{\bar{v}} \big) \ dS.\label{eq:err_simplify}
\end{align}
Here, $\mathcal{N}_1$ and $\mathcal{N}_2$ are defined in (\ref{eq_full}).

To get an acceptable error, we choose appropriate $(\underline{u}^h, \underline{v}^h)$, and for simplicity, we will assume in the following that $(u^h, v^h) = (\underline{u}^h, \underline{v}^h)$ at $t=0$. In the numerical experiments, however, we do not satisfy this condition. On each $\Omega_j$, we apply the following conditions for all times $t$ and for all $\phi \in U^q_h$ and $\psi \in V^s_h$,
\begin{equation}\label{proj}
    \int_{\Omega_j} \nabla \phi \cdot \nabla \delta_u \ d{\bf x} = 0, \ \int_{\Omega_j} \delta_u \ d{\bf x} = 0, \ 
    \int_{\Omega_j} \psi \delta_{v} \ d{\bf x} = 0.
\end{equation}
The $H^1$ projection equation for $\underline{u}^h$ and the $L^2$ projection equation for $\underline{v}^h$ are solvable by counting and uniqueness arguments. For a detailed discussion, see \cite{appelo2015new,appelo2020energy}. Since $q-2 \leq s \leq q$ is given in ~\eqref{eq:degree}, the first to fifth terms on the right side of ~\eqref{eq:err_simplify} vanish due to the special projection ~\eqref{proj}. This leads to the following equation,
\begin{equation}\label{eq:err_simplify2}
\begin{aligned}
    & \int_{\Omega_j} 2\mbox{Re}\big(\nabla \underline{e}_{\bar{u}} \cdot \nabla ( \underline{e}_u )_t \big) + 2\mbox{Re}\big( ( \underline{e}_{v} )_t \underline{e}_{\bar{v}} \big) \ d{\bf x} 
    = \int_{\Omega_j}  
    2\mbox{Re}\big( \beta({\bf x}) f(|u|^2) \underline{e}_{\bar{u}} (- u_t  + v) \big) \ d{\bf x} \\
    - & \mathcal{N}_1 + \mathcal{N}_2 + \int_{\partial \Omega_j} 2\mbox{Re}\big( {\bf n} \cdot \nabla \underline{e}_{\bar{u}} (\underline{e}_v^\ast - \delta_v^\ast - \underline{e}_v ) \big) + 2\mbox{Re}\big( {\bf n} \cdot (\nabla (\underline{e}_u - \delta_u))^\ast \underline{e}_{\bar{v}} \big) \ dS.
\end{aligned}
\end{equation}
Moreover, by taking advantage of the fact that $u_t - v = 0$, the equation ~\eqref{eq:err_simplify2} yields
\begin{equation}\label{eq:err_simplify3}
\begin{aligned}
    & \int_{\Omega_j} 2\mbox{Re}\big(\nabla \underline{e}_{\bar{u}} \cdot \nabla ( \underline{e}_u )_t \big) + 2\mbox{Re}\big( ( \underline{e}_{v} )_t \underline{e}_{\bar{v}} \big) \ d{\bf x} 
     \\
     =  & -\sum_{j,{\bf k}} 2\mbox{Re}\Big(\omega_{j,{\bf k}} \beta({\bf x}_{j,{\bf k}}) f(|u^h({\bf x}_{j,{\bf k}})|^2) \underline{e}_{\bar{u}}({\bf x}_{j,{\bf k}}) \big( (\underline{e}_u)_t({\bf x}_{j,{\bf k}})  - \underline{e}_v({\bf x}_{j,{\bf k}})\big) \Big) \\
      & + \sum_{j,{\bf k}}2\mbox{Re}\Big( \omega_{j,{\bf k}}\beta({\bf x}_{j,{\bf k}}) f(|u^h({\bf x}_{j,{\bf k}})|^2) \underline{e}_{\bar{u}}({\bf x}_{j,{\bf k}}) \big( (\delta_u)_t({\bf x}_{j,{\bf k}})  - \delta_v({\bf x}_{j,{\bf k}})\big) \Big)  \\
      &+\mathcal{N}_2+\int_{\partial \Omega_j} 2\mbox{Re}\big( {\bf n} \cdot \nabla \underline{e}_{\bar{u}} ((\underline{e}_v)^\ast - (\delta_v)^\ast - \underline{e}_v ) \big) + 2\mbox{Re}\big( {\bf n} \cdot (\nabla (\underline{e}_u - \delta_u))^\ast \underline{e}_{\bar{v}} \big) \ dS.
\end{aligned}
\end{equation}
To further simplify the equation ~\eqref{eq:err_simplify3}, we first rewrite the first term on its right side as
\begin{equation}\label{eq:auxillary}
\begin{aligned}
    & -\sum_{j,{\bf k}} 2\mbox{Re}\Big(\omega_{j,{\bf k}} \beta({\bf x}_{j,{\bf k}}) f(|u^h({\bf x}_{j,{\bf k}})|^2) \underline{e}_{\bar{u}}({\bf x}_{j,{\bf k}}) \big( (\underline{e}_u)_t({\bf x}_{j,{\bf k}})  - \underline{e}_v({\bf x}_{j,{\bf k}})\big) \Big)\\
    = &\sum_{j,{\bf k}} 2\mbox{Re}\big(\omega_{j,{\bf k}}\beta({\bf x}_{j,{\bf k}}) f(|u^h({\bf x}_{j,{\bf k}})|^2) \underline{e}_{\bar{u}}({\bf x}_{j,{\bf k}})\underline{e}_v({\bf x}_{j,{\bf k}}) \big) \\
    &- \frac{d}{dt} \sum_{j,{\bf k}} \int_0^{|\underline{e}_u({\bf x}_{j,{\bf k}})|^2} \omega_{j,{\bf k}}\beta({\bf x}_{j,{\bf k}}) f(|\underline{u}^h({\bf x}_{j,{\bf k}})|^2 + z) \ dz  \\
    &+ \sum_{j,{\bf k}} \int_0^{|\underline{e}_u({\bf x}_{j,{\bf k}})|^2} \omega_{j,{\bf k}}\beta({\bf x}_{j,{\bf k}}) \Big(\frac{d}{dt} f(|\underline{u}^h({\bf x}_{j,{\bf k}})|^2 + z)\Big) \ dz  \\
    &+\sum_{j,{\bf k}} 4\omega_{j,{\bf k}}\beta({\bf x}_{j,{\bf k}}) \left(\frac{df(w)}{dw}\Big|_{w = \widehat{w}} \mbox{Re}\big(\underline{u}^h({\bf x}_{j,{\bf k}})\underline{e}_{\bar{u}}({\bf x}_{j,{\bf k}})\big)\right)\mbox{Re}\big(\underline{e}_{\bar{u}}({\bf x}_{j,{\bf k}}) (\underline{e}_u)_t({\bf x}_{j,{\bf k}})\big),
\end{aligned}
\end{equation}
where we have used the relation
 $ |\underline{u}^h|^2 + |\underline{e}_u|^2 = |u^h|^2 + 2 \mbox{Re}(\underline{u}^h \underline{e}_{\bar{u}})$
to generate the second term on the right-hand side of ~\eqref{eq:auxillary}. In addition, $\widehat{w}$ in the fourth term of the right-hand side of ~\eqref{eq:auxillary} is given by $\widehat{w} = |\underline{u}^h({\bf x}_{j,{\bf k}})|^2 + \theta_1|\underline{e}_u({\bf x}_{j,{\bf k}})|^2$ with $\theta_1 \in (0, 1)$.  Then combining ~\eqref{eq:auxillary} with ~\eqref{eq:err_energy} and ~\eqref{eq:err_simplify3} yields
\begin{equation}\label{eq:err_simplify4}
\begin{aligned}
\frac{d\mathcal{E}}{dt} &= \sum_j \!\!\int_{\Omega_j} \!\!\!2\mbox{Re}\big(\underline{e}_{\bar{v}}(\underline{e}_{v})_t\big) + 2\mbox{Re}\big(\nabla \underline{e}_{\bar{u}} \cdot (\nabla \underline{e}_{u})_t\big)\ d{\bf x} + \sum_{j,{\bf k}}\frac{d}{dt} \int_0^{|\underline{e}_u({\bf x}_{j,{\bf k}})|^2} \!\!\!\!\!\!\!\!\!\!\!\!\!\omega_{j,{\bf k}}\beta({\bf x}_{j,{\bf k}}) f(|\underline{u}^h({\bf x}_{j,{\bf k}})|^2+z) dz \\
&= \mathcal{V} +\sum_j\int_{\partial \Omega_j} 2\mbox{Re}\big( {\bf n} \cdot \nabla \underline{e}_{\bar{u}} ((\underline{e}_v)^\ast - (\delta_v)^\ast - \underline{e}_v ) \big) + 2\mbox{Re}\big( {\bf n} \cdot (\nabla (\underline{e}_u - \delta_u))^\ast \underline{e}_{\bar{v}} \big) \ dS,
\end{aligned}
\end{equation}
where
\begin{equation}\label{volume_int}
\begin{aligned}
\mathcal{V} := & \mathcal{N}_2 
+ \sum_{j,{\bf k}} 4\omega_{j,{\bf k}}\beta({\bf x}_{j,{\bf k}}) \left(\frac{df(w)}{dw}\Big|_{w = \widehat{w}}\mbox{Re}\big(\underline{u}^h({\bf x}_{j,{\bf k}})\underline{e}_{\bar{u}}({\bf x}_{j,{\bf k}})\big)\right)\mbox{Re}\big(\underline{e}_{\bar{u}}({\bf x}_{j,{\bf k}}) (\underline{e}_u)_t({\bf x}_{j,{\bf k}})\big)\\
&+ \sum_{j,{\bf k}}2\mbox{Re}\Big( \omega_{j,{\bf k}}\beta({\bf x}_{j,{\bf k}}) f(|u^h({\bf x}_{j,{\bf k}})|^2) \underline{e}_{\bar{u}}({\bf x}_{j,{\bf k}}) \big( (\delta_u)_t({\bf x}_{j,{\bf k}})  - \delta_v({\bf x}_{j,{\bf k}})\big) \Big) \\
      & + \sum_{j,{\bf k}} 2\mbox{Re}\big(\omega_{j,{\bf k}}\beta({\bf x}_{j,{\bf k}}) f(|u^h({\bf x}_{j,{\bf k}})|^2) \underline{e}_{\bar{u}}({\bf x}_{j,{\bf k}})\underline{e}_v({\bf x}_{j,{\bf k}}) \big) \\
     & + \sum_{j,{\bf k}} \int_0^{|\underline{e}_u({\bf x}_{j,{\bf k}})|^2} \omega_{j,{\bf k}}\beta({\bf x}_{j,{\bf k}}) \Big(\frac{d}{dt} f(|\underline{u}^h({\bf x}_{j,{\bf k}})|^2 + z)\Big) dz.
      \end{aligned}
\end{equation}
By combining the effects of adjacent elements, similar to the derivation of Theorem \ref{thm1}, we can rewrite the boundary terms in equation ~\eqref{eq:err_simplify4} to obtain
\begin{equation}\label{eq:err_simplify5}
\begin{aligned}
\frac{d\mathcal{E}}{dt} &= \sum_j \!\!\int_{\Omega_j} \!\!\!2\mbox{Re}\big(\underline{e}_{\bar{v}}(\underline{e}_{v})_t\big) + 2\mbox{Re}\big(\nabla \underline{e}_{\bar{u}} \cdot (\nabla \underline{e}_{u})_t\big)\ d{\bf x} + \sum_{j,{\bf k}}\frac{d}{dt} \int_0^{|\underline{e}_u({\bf x}_{j,{\bf k}})|^2} \!\!\!\!\!\!\!\!\!\!\!\!\!\omega_{j,{\bf k}}\beta({\bf x}_{j,{\bf k}}) f(|\underline{u}^h({\bf x}_{j,{\bf k}})|^2+z) dz \\
&= \mathcal{V}  - \sum_j\int_{F_j} 2 \mbox{Re} \big(\gamma \big|[[\underline{e}_v]]\big|^2 + \tau \big|[[\nabla \underline{e}_{u}]]\big|^2\big)+ 2 \mbox{Re} \big([[\nabla \underline{e}_{\bar{u}}]] \delta_v^\ast + (\nabla \delta_{u})^\ast \cdot [[\underline{e}_{\bar{v}}]]\big)\  dS,
\end{aligned}
\end{equation}
where, again, $F_j$ represents the interelement boundaries, and $\mathcal{V}$ is defined in ~\eqref{volume_int}.

In the following discussion, let $C$ be a universal constant independent of the diameter $h$ of a shape-regular mesh, although its value may vary from line to line. Sobolev norms are denoted by $\|\cdot\|$, while seminorms are denoted by $\vert \cdot \vert$. We assume that the solution is smooth enough up to some time $T$ and that $f(w)$, $\beta({\bf x})f(w)$, and $\beta({\bf x})\frac{df}{dw} (w)$ are bounded. While it is sufficient that these bounds hold in some neighbourhood of the actual values of $u$ up to time $T$, for simplicity, we will assume that they hold for all $w$. To simplify the statement of the primary theorem, we assume that $s > \frac{d}{2} - 1$, which implies that the functions in $H^{s+1}(\Omega)$ are continuous. Moreover, since $s \leq q$, all functions in $H^{q+1}(\Omega)$ are also continuous. Given these assumptions, the following error estimate holds.

\begin{truth}\label{truth2}
Suppose $f(w)$ is a smooth bounded function that satisfies the lower bound ~\eqref{eq:lower} and the Assumption \ref{assump} holds. Let $\widehat{q} = \min(q-1, s)$, where $q-2 \leq s \leq q$. Assuming that $\|u\|_{L^\infty([0, T], H^{q+1}(\Omega))}$ and $\|v\|_{L^\infty([0, T], H^{s+1}(\Omega))}$ are bounded, along with $\beta({\bf x})f(w)$, $f(w)$, and $\beta({\bf x})\frac{df(w)}{dw}$, there exist constants $C_0$ and $C_1$ which depend only on $s$, $q$, $\gamma$, $\tau$, and the shape regularity of the mesh, but are independent of $h$, such that
\begin{equation}\label{eq:thm2}
\|\nabla e_u(\cdot, t)\|_{L^2(\Omega)}^2 + \|e_v(\cdot, t)\|_{L^2(\Omega)}^2 \leq C_0e^{C_1t}h^{2\zeta}\quad \forall\ 0\leq t\leq T,
\end{equation}
where $\zeta = \widehat{q}$ when $\gamma,\tau \geq 0$, and $\zeta = \widehat{q} + \frac{1}{2}$ when $\gamma,\tau > 0$. 


\end{truth}

\begin{proof}
From the Bramble-Hilbert lemma (e.g., \cite{ciarlet2002finite}), we have
\begin{equation}\label{lemma_BH}
\begin{aligned}
&\|\delta_u\|_{L^2(\Omega_j)}^2 \leq Ch^{2q+2}|u(\cdot, t)|_{H^{q+1}(\Omega_j)}^2,\quad \|\delta_v\|_{L^2(\Omega_j)}^2 \leq Ch^{2s+2} |v(\cdot, t)|_{H^{s+1}(\Omega_j)}^2,\\
& \Big\Vert \frac{\partial \delta_u}{\partial t}\Big\Vert_{H^{q+1}(\Omega_j)} \leq Ch^{2q+2}\Big| \frac{\partial u(\cdot, t)}{\partial t}\Big|_{H^{q+1}(\Omega_j)}.
\end{aligned}
\end{equation}
We also point out that since our approximations are piecewise polynomials and the assumption that $u$ and $v$ are bounded, we can establish the following inequalities,
\[\|\underline{u}^h\|_{L^\infty(\Omega_j)} \leq C, \quad \Big\Vert \frac{\partial \underline{u}^h}{\partial t}\Big\Vert_{L^\infty(\Omega_j)} \leq C.\]

Starting with the estimation of $\mathcal{N}_2$ in ~\eqref{volume_int}, we use the mean value theorem for $f(w)$, together with the Assumption \ref{assump}, the inequalities ~\eqref{lemma_BH} and the Cauchy-Schwartz inequality to obtain
\begin{equation}\label{eq:auxiliary2}
\begin{aligned}
\mathcal{N}_2
    &=  \sum_{j,{\bf k}} 2\mbox{Re}\Bigg[\omega_{j,{\bf k}}\beta({\bf x}_{j,{\bf k}}) \Bigg( \frac{d f(w)}{dw}\Big|_{w = \widetilde{w}} \Big( |e_u({\bf x}_{j,{\bf k}})|^2 - 2 \mbox{Re} \big(u({\bf x}_{j,{\bf k}}) e_{\bar{u}}({\bf x}_{j,{\bf k}})\big) \Big) u({\bf x}_{j,{\bf k}}) \\
    & \hspace{0.6cm}- f(|u^h({\bf x}_{j,{\bf k}})|^2) \big(\underline{e}_u({\bf x}_{j,{\bf k}}) - \delta_u({\bf x}_{j,{\bf k}})\big) \Bigg) \underline{e}_{\bar{v}}({\bf x}_{j,{\bf k}}) \Bigg]\\
    \leq & \sum_{j,{\bf k}} 2\omega_{j,{\bf k}} \mbox{max}\Big|\beta\frac{df(w)}{dw}\Big|\mbox{max}|u| \big|\big(\underline{e}_{\bar{u}}({\bf x}_{j,{\bf k}}) - \delta_{\bar{u}}({\bf x}_{j,{\bf k}})\big)\big(\underline{e}_{u}({\bf x}_{j,{\bf k}}) - \delta_{u}({\bf x}_{j,{\bf k}})\big)\underline{e}_{v}({\bf x}_{j,{\bf k}})\big| \\
    &\hspace{0.6cm}+ 2\omega_{j,{\bf k}}\mbox{max}\Big|\beta \frac{df(w)}{dw}\Big|\Big(\mbox{max}|u^2||\underline{e}_{\bar{u}}({\bf x}_{j,{\bf k}})\underline{e}_{v}({\bf x}_{j,{\bf k}})|  + \mbox{max}|u|^2 |\underline{e}_{\bar{u}}({\bf x}_{j,{\bf k}})\underline{e}_{v}({\bf x}_{j,{\bf k}})|\Big)\\
    &\hspace{0.6cm}+ 2\omega_{j,{\bf k}}\mbox{max}\Big|\beta \frac{df(w)}{dw}\Big|\Big(\mbox{max}|u^2||\delta_{\bar{u}}({\bf x}_{j,{\bf k}})\underline{e}_{v}({\bf x}_{j,{\bf k}})|  + \mbox{max}|u|^2 |\delta_{\bar{u}}({\bf x}_{j,{\bf k}})\underline{e}_{v}({\bf x}_{j,{\bf k}})|\Big)\\
    &\hspace{0.6cm} + 2\omega_{j,{\bf k}} \mbox{max}|\beta f(w)| \Big(|\underline{e}_u({\bf x}_{j,{\bf k}})\underline{e}_{v}({\bf x}_{j,{\bf k}})| + |\delta_{u}({\bf x}_{j,{\bf k}})\underline{e}_{v}({\bf x}_{j,{\bf k}})|\Big)\\
    \leq & C\big(1+h^{q+1}|u(\cdot,t)|_{H^{q+1}(\Omega)} \big)\mathcal{E}  + C (h^{q+1} + h^{2q+2}) \sqrt{\mathcal{E}} |u(\cdot,t)|_{H^{q+1}(\Omega)},
\end{aligned}
\end{equation}
where $\widetilde{w} = |u({\bf x}_{j,{\bf k}})|^2 + \theta_2 |e_u({\bf x}_{j,{\bf k}})|^2$ for some $\theta_2 \in (0,1)$. Note that the first equality in ~\eqref{eq:auxiliary2} applies
\[|u^h|^2 - |u|^2 = |e_u|^2 - 2\mbox{Re}(ue_{\bar{u}})\quad  \mbox{and} \quad e_u = \underline{e}_{u} - \delta_u.\]

Similar as above, we can estimate the third to fifth nonlinear volume integrals on the right side of ~\eqref{volume_int} as
\begin{equation}\label{eq:auxillary7}
\begin{aligned}
&\sum_{j,{\bf k}} 2\mbox{Re}\Big( \omega_{j,{\bf k}}\beta({\bf x}_{j,{\bf k}}) f(|u^h({\bf x}_{j,{\bf k}})|^2) \underline{e}_{\bar{u}}({\bf x}_{j,{\bf k}}) \big( (\delta_u)_t({\bf x}_{j,{\bf k}})  - \delta_v({\bf x}_{j,{\bf k}})\big) \Big) \\
+&\sum_{j,{\bf k}} 2\mbox{Re}\big(\omega_{j,{\bf k}}\beta({\bf x}_{j,{\bf k}}) f(|u^h({\bf x}_{j,{\bf k}})|^2) \underline{e}_{\bar{u}}({\bf x}_{j,{\bf k}})\underline{e}_v({\bf x}_{j,{\bf k}}) \big)  + \int_0^{|\underline{e}_u({\bf x}_{j,{\bf k}})|^2} \omega_{j,{\bf k}}\beta({\bf x}_{j,{\bf k}}) \Big(\frac{d}{dt} f(|\underline{u}^h({\bf x}_{j,{\bf k}})|^2 + z)\Big) dz  \\
\leq & \sum_{j,{\bf k}} 2\omega_{j,{\bf k}}\mbox{max}\Big| \beta f(w) \Big|\Big( \big|\underline{e}_{{u}}({\bf x}_{j,{\bf k}}) (\delta_u)_t({\bf x}_{j,{\bf k}})\big|  + \big|\underline{e}_{{u}}({\bf x}_{j,{\bf k}})\delta_v({\bf x}_{j,{\bf k}})\big| \Big)  \\
 + &\sum_{j,{\bf k}}\omega_{j,{\bf k}}\mbox{max}\Big|\beta\frac{df(w)}{dw}\Big|\mbox{max}\Big|\frac{d|\underline{u}^h|^2}{dt}\Big||\underline{e}_u({\bf x}_{j,{\bf k}})|^2 + \sum_{j,{\bf k}} 2\omega_{j,{\bf k}}\mbox{max}\big|\beta f(w)\big| |\underline{e}_{{u}}({\bf x}_{j,k})\underline{e}_v({\bf x}_{j,{\bf k}})|\\
\leq & C\mathcal{E} + Ch^{\min\{q+1, s+1\}}\sqrt{\mathcal{E}}\left(\Big|\frac{\partial u(\cdot, t)}{\partial t}\Big|_{H^{q+1}(\Omega_j)} + |v(\cdot, t)|_{H^{s+1}(\Omega_j)}\right).
\end{aligned}
\end{equation}

To evaluate the second nonlinear volume integral on the right side of ~\eqref{volume_int}, which includes $\underline{e}_{\bar{u}}(\underline{e}_u)_t$, it's necessary to find an upper bound for $\|(\underline{e}_u)_t\|_{L^2(\Omega_j)}$. To do this, we examine the EDG formulation ~\eqref{eq:21} with the test function $\phi$ replaced by $\phi = (e_{\bar{u}})_t - e_{\bar{v}}$, i.e.,
\begin{equation}\label{eq:est_et1}
\begin{aligned}
&\int_{\Omega_j} \nabla \big((e_{\bar{u}})_t - e_{\bar{v}}\big) \cdot \nabla (u^h_t - v^h)\ d{\bf x} \\
&+ \sum_{j,{\bf k}}\omega_{j,{\bf k}}\beta({\bf x}_{j,{\bf k}}) f(|u^h({\bf x}_{j,{\bf k}})|^2) \big((e_{\bar{u}})_t({\bf x}_{j, {\bf k}}) - e_{\bar{v}}({\bf x}_{j,{\bf k}})\big) (u^h_t({\bf x}_{j,{\bf k}}) - v^h({\bf x}_{j,{\bf k}}) \\
= &\int_{\partial \Omega_j} {\bf n} \cdot \nabla \big((e_{\bar{u}})_t - e_{\bar{v}}\big) (v^\ast - v^h)\ dS.
\end{aligned}
\end{equation}
With the fact that $u_t - v = 0$ given in ~\eqref{system}, we can rewrite equation ~\eqref{eq:est_et1} into
\begin{equation}\label{eq:est_et}
\begin{aligned}
&\int_{\Omega_j} \nabla \big((e_{\bar{u}})_t - e_{\bar{v}}\big) \cdot \nabla \big((e_{u})_t - e_{v}\big)\ d{\bf x} \\
&+ \sum_{j,{\bf k}}\omega_{j,{\bf k}}\beta({\bf x}_{j,{\bf k}}) f(|u^h({\bf x}_{j,{\bf k}})|^2) \big((e_{\bar{u}})_t({\bf x}_{j, {\bf k}}) - e_{\bar{v}}({\bf x}_{j,{\bf k}})\big) \big((e_{u})_t({\bf x}_{j,{\bf k}}) - e_{v}({\bf x}_{j,{\bf k}})\big) \\
= &\int_{\partial \Omega_j} {\bf n} \cdot \nabla \big((e_{\bar{u}})_t - e_{\bar{v}}\big) (e_v^\ast - e_v^h)\ dS.
\end{aligned}
\end{equation}
Furthermore, by invoking the condition ~\eqref{eq:lower}, we get
\begin{equation}\label{eq:auxiliary3}
\|\nabla \big((e_{u})_t - e_{v}\big) \|_{L^2(\Omega_j)}^2
\leq Ch^{-1} \|\nabla \big((e_{u})_t - e_{v}\big) \|_{L^2(\Omega_j)}\Big(\|e_v^\ast\|_{L^2(\Omega_j)} + \|e_v^h\|_{L^2(\Omega_j)}\Big),
\end{equation}
where we used the trace inequality to estimate the boundary integrals. On the other hand, using the first equation in ~\eqref{system} and ~\eqref{eq:mean_value}, we have
$\int_{\Omega_j} (e_{u})_t - e_v\ d{\bf x} = 0,$ which implies Poincar'{e} inequality
\[\|\nabla \big((e_{u})_t - e_{v}\big) \|_{L^2(\Omega_j)} \geq Ch^{-1} \|(e_{u})_t - e_{v} \|_{L^2(\Omega_j)}.\]
Substituting this into ~\eqref{eq:auxiliary3}, we get that
\begin{equation}\label{eq:auxiliary4}
\|(e_{u})_t - e_{v} \|_{L^2(\Omega_j)}
\leq  C\Big(\|e_v^\ast\|_{L^2(\Omega_j)} + \|e_v\|_{L^2(\Omega_j)}\Big).
\end{equation}
Using the definition of numerical fluxes in ~\eqref{eq:generalfl}, the relation ~\eqref{eq:error_relation}, and ~\eqref{eq:auxiliary4}, we arrive at
\begin{equation}\label{eq:bound_et}
\begin{aligned}
\|(\underline{e}_u)_t\|_{L^2(\Omega_j)} &
\leq C\Big(\|\underline{e}_v\|_{L^2(\Omega_j)} + \|\nabla \underline{e}_u\|_{L^2(\Omega_j)} \\
& + h^{s+1}|v(\cdot, t)|_{H^{s+1}(\Omega_j)} + h^{q}|u(\cdot, t)|_{H^{s+1}(\Omega_j)} + h^{q+1}\big|\frac{\partial u(\cdot, t)}{\partial t}\big|_{H^{q+1}(\Omega_j)}\Big).
\end{aligned}
\end{equation}

With the upper bound of $\|(\underline{e}_u)_t\|_{L^2(\Omega_j)}$ derived in ~\eqref{eq:bound_et}, we are now ready to estimate the second nonlinear volume integral on the right-hand side of ~\eqref{volume_int}. We use the H\"{o}lder inequality to obtain
\begin{equation}\label{eq:auxiliary8}
\begin{aligned}
&\sum_{j,{\bf k}} 4\omega_{j,{\bf k}}\beta({\bf x}_{j,{\bf k}}) \left(\frac{df(w)}{dw}\Big|_{w = |\underline{u}^h({\bf x}_{j,{\bf k}})|^2 + \theta_1|\underline{e}_u({\bf x}_{j,{\bf k}})|^2}\mbox{Re}\big(\underline{u}^h({\bf x}_{j,{\bf k}})\underline{e}_{\bar{u}}({\bf x}_{j,{\bf k}})\big)\right)\mbox{Re}\big(\underline{e}_{\bar{u}}({\bf x}_{j,{\bf k}}) (\underline{e}_u)_t({\bf x}_{j,{\bf k}})\big)\\
\leq &\sum_{j,{\bf k}} \mbox{max}\Big|\beta\frac{df(w)}{dw}\Big|\mbox{max}|\underline{u}^h|\Big(3|\underline{e}_u({\bf x}_{j,{\bf k}})\underline{e}_{\bar{u}}({\bf x}_{j,{\bf k}})(\underline{e}_{\bar{u}})_t({\bf x}_{j,{\bf k}})| + |\underline{e}_u({\bf x}_{j,{\bf k}})\underline{e}_u({\bf x}_{j,{\bf k}})(\underline{e}_{\bar{u}})_t({\bf x}_{j,{\bf k}})|\Big)\\
\leq &C \mathcal{E},
\end{aligned}
\end{equation}
where we used the the following inequality to derive the second inequality when $\underline{e}_u$ is small,
\begin{equation}\label{eq:auxillary6}
\int_{\Omega_j} |\underline{e}_u|^4\ d{\bf x} \leq C \int_{\Omega_j} |\underline{e}_u|^2\ d{\bf x}.
\end{equation}
 By combining ~\eqref{eq:err_simplify5}, ~\eqref{eq:auxiliary2}--\eqref{eq:auxillary7} and ~\eqref{eq:auxiliary8}, we have
\begin{equation}\label{eq:err_simplify6}
\begin{aligned}
\frac{d\mathcal{E}}{dt} \leq & C\Big(1+h^{q+1}|u(\cdot,t)|_{H^{q+1}(\Omega)} \Big)\mathcal{E}  \\
& + Ch^{\min\{q+1, s+1\}}\sqrt{\mathcal{E}}\left(\Big|\frac{\partial u(\cdot, t)}{\partial t}\Big|_{H^{q+1}(\Omega)} + |v(\cdot, t)|_{H^{s+1}(\Omega)} + |u(\cdot, t)|_{H^{q+1}(\Omega)}\right)\\
& + \sum_j\int_{F_j} -2 \mbox{Re} \big(\gamma\big|[[\underline{e}_v]]\big|^2 + \tau \big|[[\nabla \underline{e}_{u}]]\big|^2\big) + 2 \mbox{Re} \big([[\nabla \underline{e}_{\bar{u}}]] \delta_v^\ast + (\nabla \delta_{u})^\ast \cdot [[\underline{e}_{\bar{v}}]]\big) \ dS.
\end{aligned}
\end{equation}

To complete the estimates, we apply the method in \cite{appelo2015new} to evaluate the boundary integrals. We divide our arguments into the following cases.

\textbf{Case I}: $\gamma=0$ or $\tau=0$. In this case, we have
\begin{equation}\label{eq:bdry1}
\begin{aligned}
 & \sum_j\int_{F_j} -2 \mbox{Re} \big(\gamma \big|[[\underline{e}_v]]\big|^2 + \tau \big|[[\nabla \underline{e}_{u}]]\big|^2\big) + 2 \mbox{Re} \big([[\nabla \underline{e}_{\bar{u}}]] \delta_v^\ast + (\nabla \delta_{u})^\ast \cdot [[\underline{e}_{\bar{v}}]]\big) \ dS\\
 \leq & \sum_j\int_{F_j} 2 \mbox{Re} \big([[\nabla \underline{e}_{\bar{u}}]] \delta_v^\ast + (\nabla \delta_{u})^\ast \cdot [[\underline{e}_{\bar{v}}]]\big) \ dS 
 \leq  Ch^{\min\{s, q-1\}} \sqrt{\mathcal{E}}\left(|v(\cdot, t)|_{H^{s+1}(\Omega)} + |u(\cdot, t)|_{H^{q+1}(\Omega)}\right).
\end{aligned}
\end{equation}
Combining ~\eqref{eq:err_simplify6} with ~\eqref{eq:bdry1}, we get
\begin{equation*}
\begin{aligned}
\frac{d\mathcal{E}}{dt} \leq & C\Big(1+h^{q+1}|u(\cdot,t)|_{H^{q+1}(\Omega)} \Big)\mathcal{E} \\
 &+ Ch^{\min\{q-1, s\}}\sqrt{\mathcal{E}}\left(\Big|\frac{\partial u(\cdot, t)}{\partial t}\Big|_{H^{q+1}(\Omega)} + |v(\cdot, t)|_{H^{s+1}(\Omega)} + |u(\cdot, t)|_{H^{q+1}(\Omega)}\right).
\end{aligned}
\end{equation*}
Solving the above ordinary differential equation yields
\begin{equation}\label{eq:est1}
\mathcal{E}(T) \leq C_0e^{C_1t}h^{2\min\{q-1,s\}}.
\end{equation}

\textbf{Case II}: $\gamma, \tau > 0$. In this case, we have
\begin{equation}\label{eq:bdry2}
\begin{aligned}
 & \sum_j\int_{F_j} -2 \mbox{Re} \big(\gamma \big|[[\underline{e}_v]]\big|^2 + \tau \big|[[\nabla \underline{e}_{u}]]\big|^2\big) + 2 \mbox{Re} \big([[\nabla \underline{e}_{\bar{u}}]] \delta_v^\ast + (\nabla \delta_{u})^\ast \cdot [[\underline{e}_{\bar{v}}]]\big) \ dS\\
 \leq & \sum_j\int_{F_j} -2 \mbox{Re} \big(\gamma \big|[[\underline{e}_v]]\big|^2 + \tau \big|[[\nabla \underline{e}_{u}]]\big|^2\big)\ dS + \sum_j \int_{F_j} 2\mbox{Re}\left( \frac{\tau}{2}\big|[[\nabla \underline{e}_{u}]]\big|^2 + \frac{1}{2\tau}\big|\delta_v^\ast\big|^2 \right) \ dS\\
 &+ \sum_j\int_{F_j}2\mbox{Re}\left(\frac{\gamma}{2}\big|[[\underline{e}_{v}]]\big|^2 + \frac{1}{2\gamma}\big|(\nabla \delta_u)^\ast\big|^2\right) \ dS \\
 \leq & \sum_j\int_{F_j} \frac{1}{\tau}\mbox{Re}\big(\big|\delta_v^\ast\big|^2\big) + \frac{1}{\gamma} \mbox{Re}\big(\big|(\nabla\delta_u)^\ast\big|^2\big)\ dS \\
 \leq  &Ch^{\min\{2s+1, 2q-1\}}\left(|v(\cdot, t)|_{H^{s+1}(\Omega)}^2 + |u(\cdot, t)|_{H^{q+1}(\Omega)}^2\right).
\end{aligned}
\end{equation}
Substituting ~\eqref{eq:bdry2} into ~\eqref{eq:err_simplify6}, we get
\begin{equation*}
\begin{aligned}
\frac{d\mathcal{E}}{dt} \leq & C\Big(1+h^{q+1}|u(\cdot,t)|_{H^{q+1}(\Omega)} \Big)\mathcal{E} + Ch^{\min\{2s+1, 2q-1\}}\left(|v(\cdot, t)|_{H^{s+1}(\Omega)}^2 + |u(\cdot, t)|_{H^{q+1}(\Omega)}^2\right)\\
&+ Ch^{\min\{q+1, s+1\}}\sqrt{\mathcal{E}}\left(\Big|\frac{\partial u(\cdot, t)}{\partial t}\Big|_{H^{q+1}(\Omega)} + |v(\cdot, t)|_{H^{s+1}(\Omega)} + |u(\cdot, t)|_{H^{q+1}(\Omega)}\right),
\end{aligned}
\end{equation*}
which gives
\begin{equation}\label{eq:est2}
\mathcal{E}(T) \leq C_0e^{C_1t}h^{2\min\{q-\frac{1}{2},s + \frac{1}{2}\}}.
\end{equation}
Since $\underline{e}_u = e_u + \delta_u, \underline{e}_v = e_v + \delta_v$ as given in ~\eqref{eq:error_relation}, \eqref{eq:thm2} follows from the triangle inequality, ~\eqref{eq:est1} and ~\eqref{eq:est2}. This completes the proof.
\end{proof}

\begin{remark}\label{remark1}
If $\beta(\mathbf{x})f(|u|^2) < 0$ for some $u$, a new variable $u = e^{-i\frac{\alpha}{2}t}\eta$ can be introduced to transform the original problem. This results in the modified problem
\[\eta_{tt} - \Delta \eta + \left(\frac{\alpha^2}{4} + \beta({\bf x})f(|\eta|^2)\right)\eta = 0.\]
By introducing another new variable $\eta = e^{\theta t} w, \theta > 0$, and using the EDG scheme to solve for $w$, we can obtain error estimates provided that $\theta^2 + \frac{\alpha^2}{4} + \beta(\mathbf{x})f (|\eta|^2)$ is positive.
\end{remark}

\begin{remark}
Similar to the discussion in \cite{appelo2015new}, it is possible to obtain an improved error estimate of $h^{q}$ when $q = s + 1$ in both one-dimensional and multidimensional problems with Cartesian grids by constructing $(\underline{u}^h, \underline{v}^h)$ such that the boundary terms in ~\eqref{eq:err_simplify4} vanish, provided that the flux parameters in ~\eqref{eq:generalfl} satisfy $\mu(1 - \mu) = \gamma\tau$.
\end{remark}

\begin{remark}
The analysis presented in Section \ref{sec_formula} and Section \ref{sec:error} focuses only on the case where periodic boundary conditions are used. However, for problems with physical boundary conditions that satisfy the condition
\[a\frac{\partial u({\bf x}, t)}{\partial t} + b\nabla u({\bf x}, t)\cdot{\bf n} = 0, \quad {\bf x}\in \partial\Omega,\]
where $a^2 + b^2 = 1$ and $a\geq 0, b\geq 0$, the technique used in \cite{appelo2015new} can be used to specify the numerical fluxes at physical boundaries while preserving the energy and error estimates.
\end{remark}

\begin{remark}
It is worth mentioning that the error estimation given in the previous sections seems to be too conservative for the problems studied in the numerical experiments section. In that section, we observe that the error growth is no worse than linear with respect to time.
\end{remark}

\section{Numerical simulations}
\label{sec:numerical}

In this section, we perform numerical experiments to demonstrate and verify the convergence of the EDG scheme proposed in Section \ref{sec_formula}. The experiments are performed in both one and two dimensions, and we use a standard modal basis formulation. For two-dimensional problems, a square domain is used, and tensor-product Legendre polynomials are used as basis functions on a square reference element. To evaluate the nonlinear integral in the scheme ~\eqref{eq:21}--\eqref{eq:22}, tensor-product Gauss rules with $q+1$ nodes in each coordinate on the reference element are used. In all experiments, we use the fourth-order Runge-Kutta scheme for time stepping and set the flux splitting parameter to $\xi=1$ when using the Sommerfeld fluxes ~\eqref{flux:up}. To achieve the desired convergence rate for the spatial discretization in all experiments, we choose a sufficiently small time step size such that $\Delta t = \frac{0.00975}{\pi} h$
to ensure that the temporal error is dominated by the spatial error. 

\subsection{One-dimensional case}
\label{liner_sec}

We start with some numerical examples in the one-dimensional case for different versions of model~\eqref{EQ:NLSW}.

\subsubsection{Example I}

We first consider model~\eqref{EQ:NLSW} with $\alpha = \beta = 1$ and $f(|u|^2) = 1$, that is,
\begin{align}\label{eq:51}
\begin{split}
    &u_{tt} - u_{xx} + i u_t + u =0, \quad x \in (0, 2\pi),\\
    & u(x,0) = e^{ix}, \quad u_t (x,0) = i e^{ix}\,.
\end{split}
\end{align}
with the periodic boundary $u(0,t)=u(2\pi,t)$. This is a linear problem and it has the following exact solution
\begin{equation}\label{sol:eq51}
    u(x,t) = e^{i(x+t)}.
\end{equation}

To discretize the spatial interval, we use uniform spacing with vertices $x_j = jh$ for $j = 0,\cdots, N$ and $h = 2\pi/N$. In our studies, we present the results by considering the degree of the approximation space of $u^h$ and $v^h$ as $q$ and $s$, respectively. Specifically, we consider two cases: $u^h$ and $v^h$ are in the same approximation space with $s = q = (1,2,3,4)$; and $u^h$ and $v^h$ are in the different approximation space with $s = q - 1 = (0,1,2,3)$. Furthermore, we test three different fluxes: the central flux ~\eqref{flux:central}, which we denote as C.-flux, the alternating flux ~\eqref{flux:a1} with $\mu=0$, which we denote as A.-flux, and the Sommerfeld flux ~\eqref{flux:up}, which we denote as S.-flux. It is important to note that both the C.-flux and the A.-flux are energy-conserving methods, while the S.-flux is an energy-dissipating method. We want to emphasize that the performance of $\mu=1$ is similar to $\mu=0$ in the case of the A.-flux; therefore, we only report the results for $\mu=0$ in the remaining sections of the paper.

\begin{figure}[!htb]
\centering
{\includegraphics[width=0.32\textwidth]{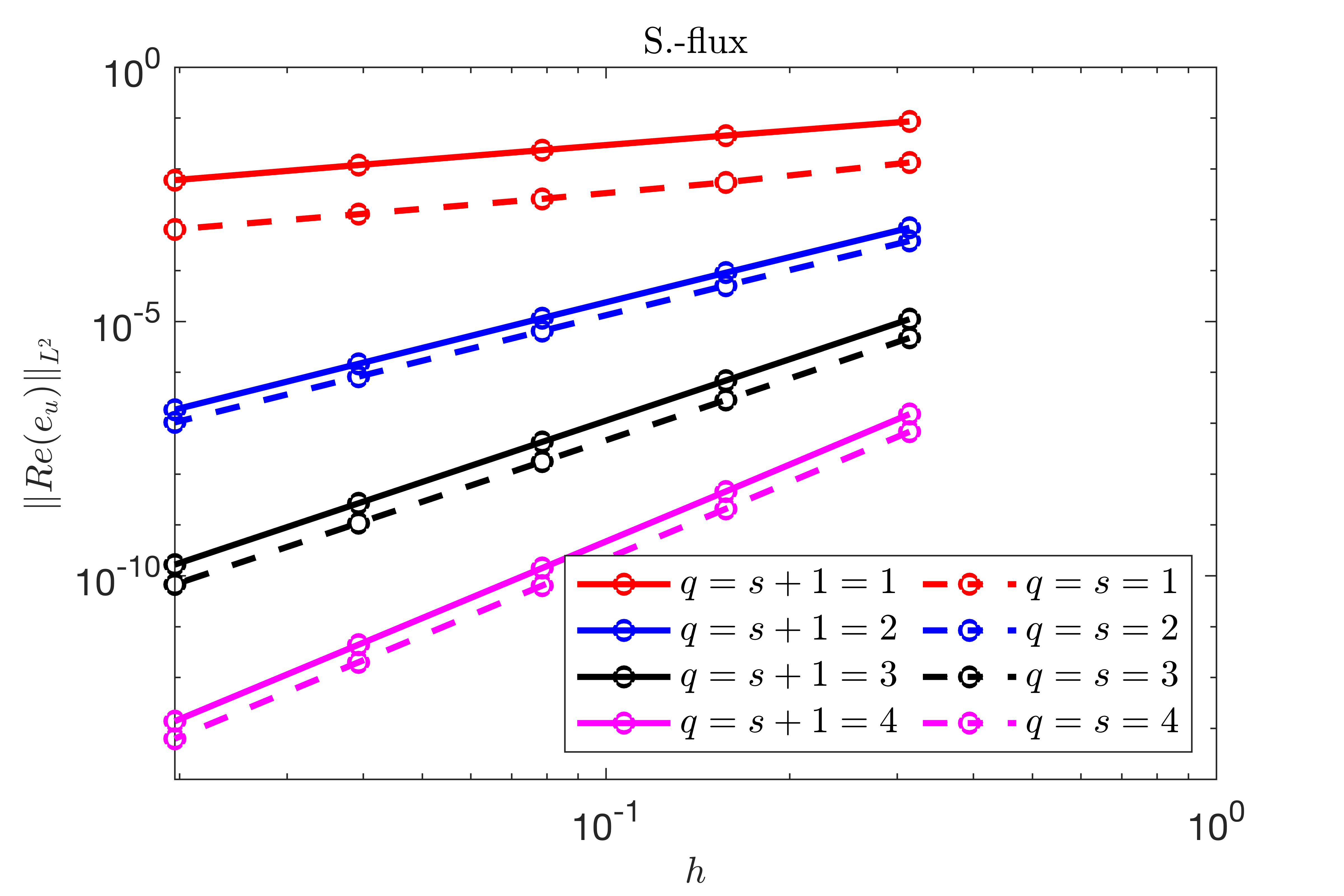}}
{\includegraphics[width=0.32\textwidth]{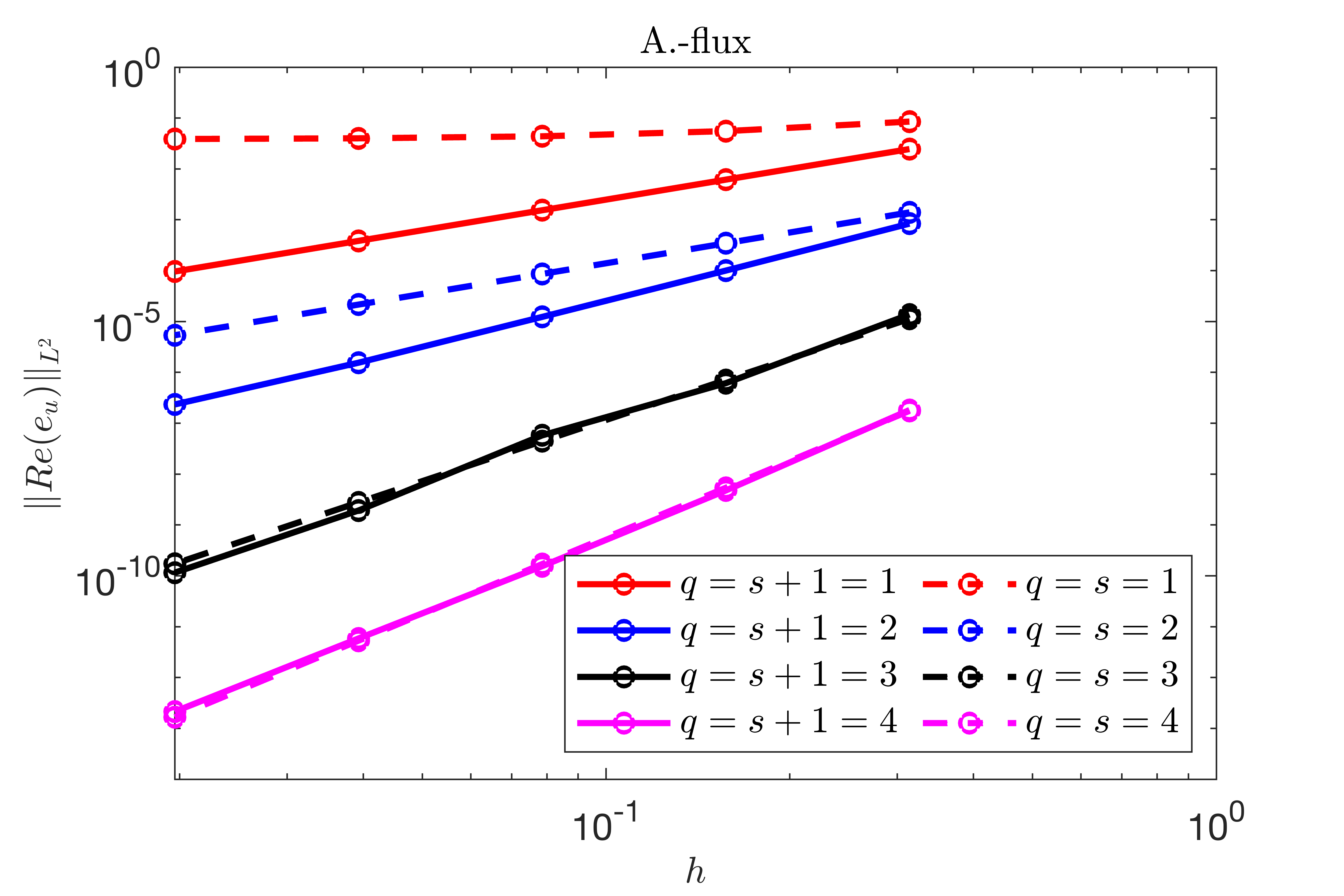}} 
{\includegraphics[width=0.32\textwidth]{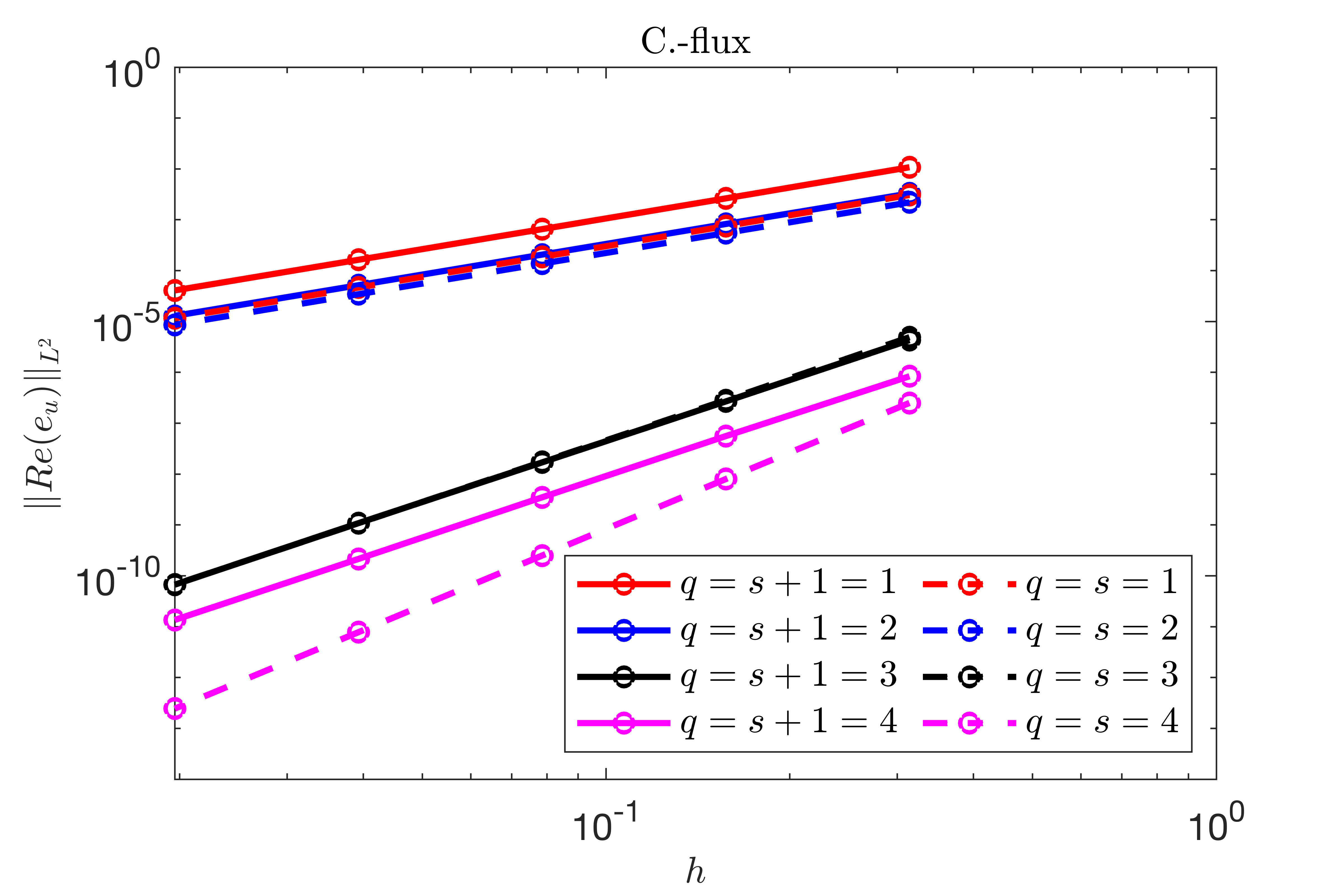}} \\
{\includegraphics[width=0.32\textwidth]{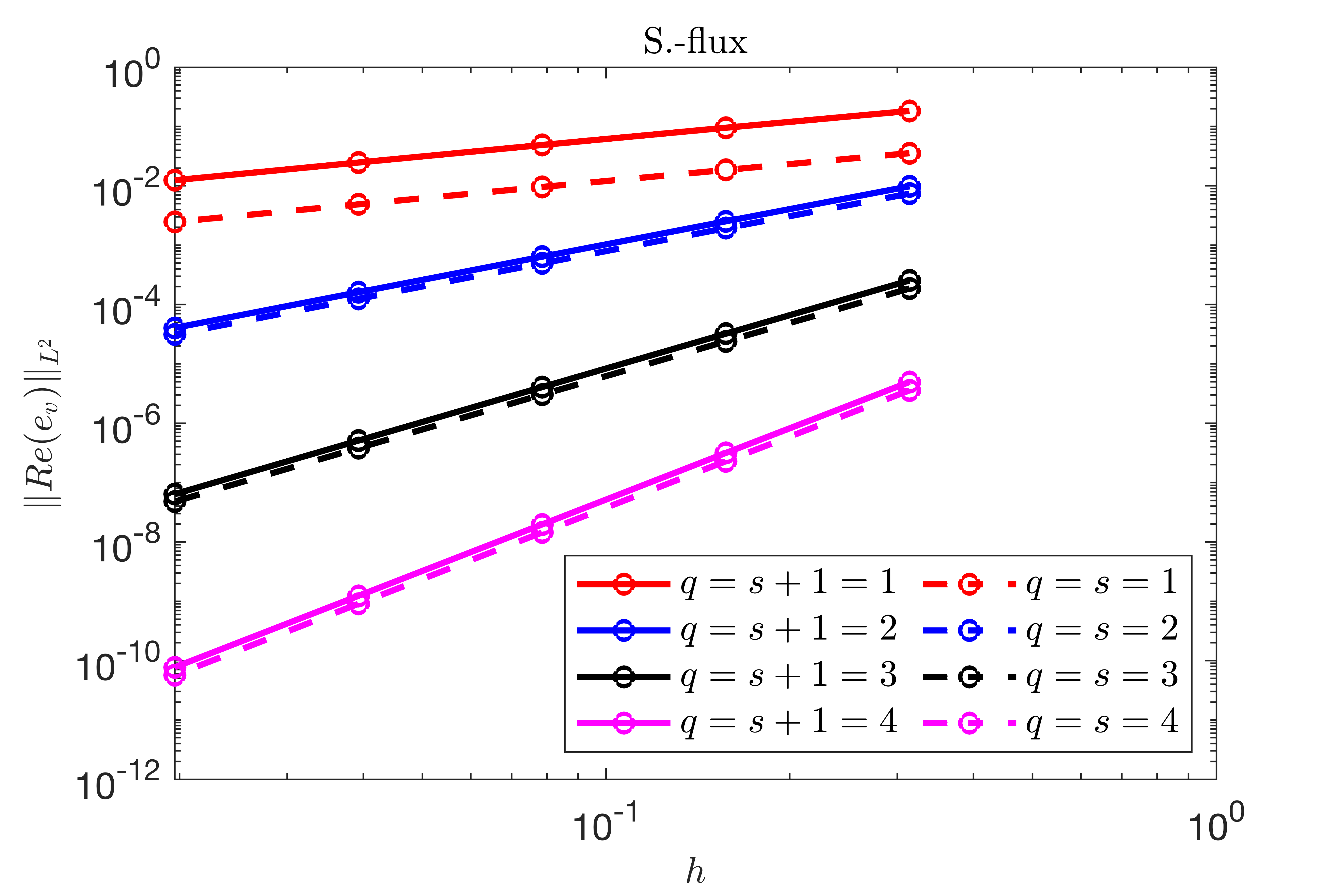}} 
{\includegraphics[width=0.32\textwidth]{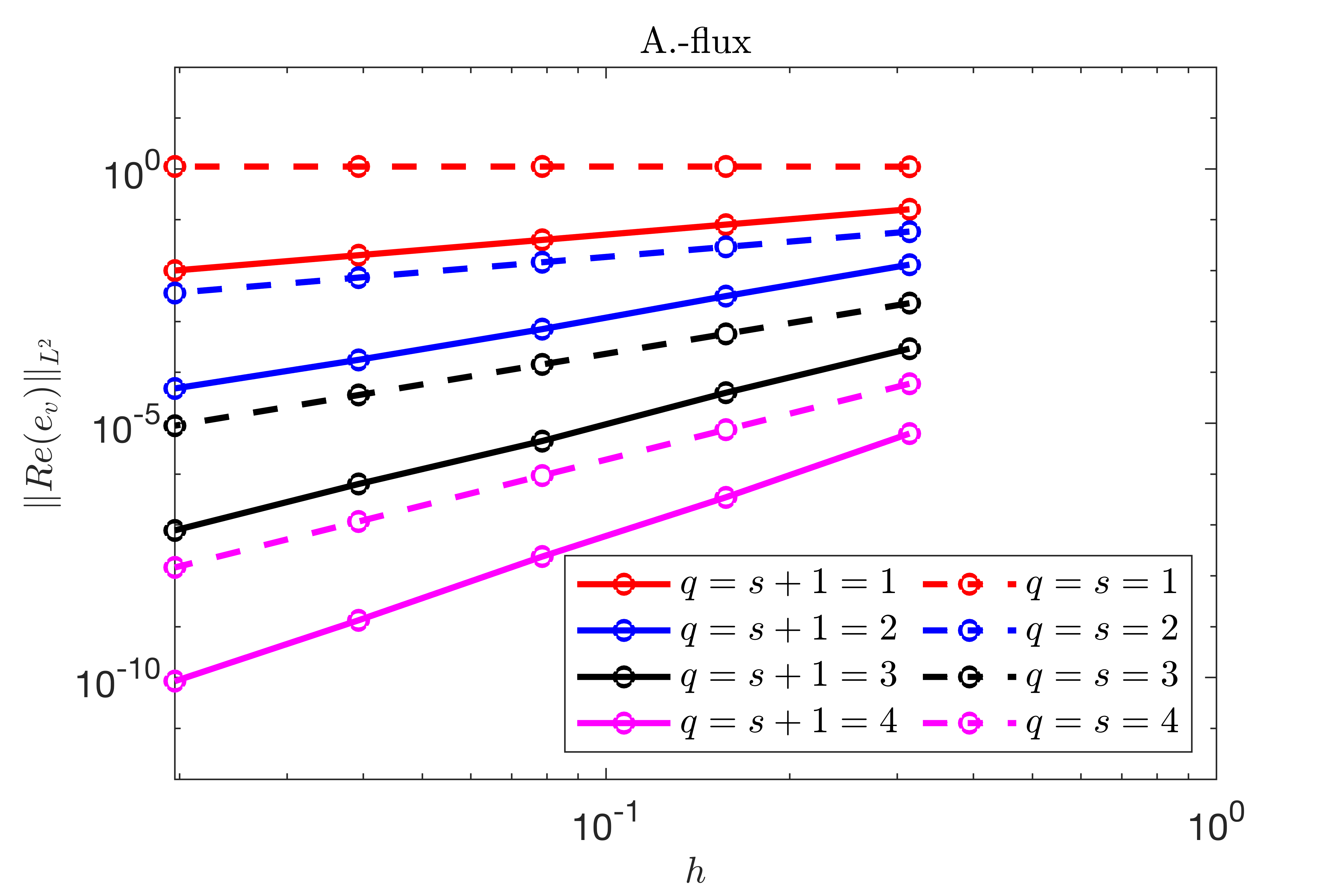}}
{\includegraphics[width=0.32\textwidth]{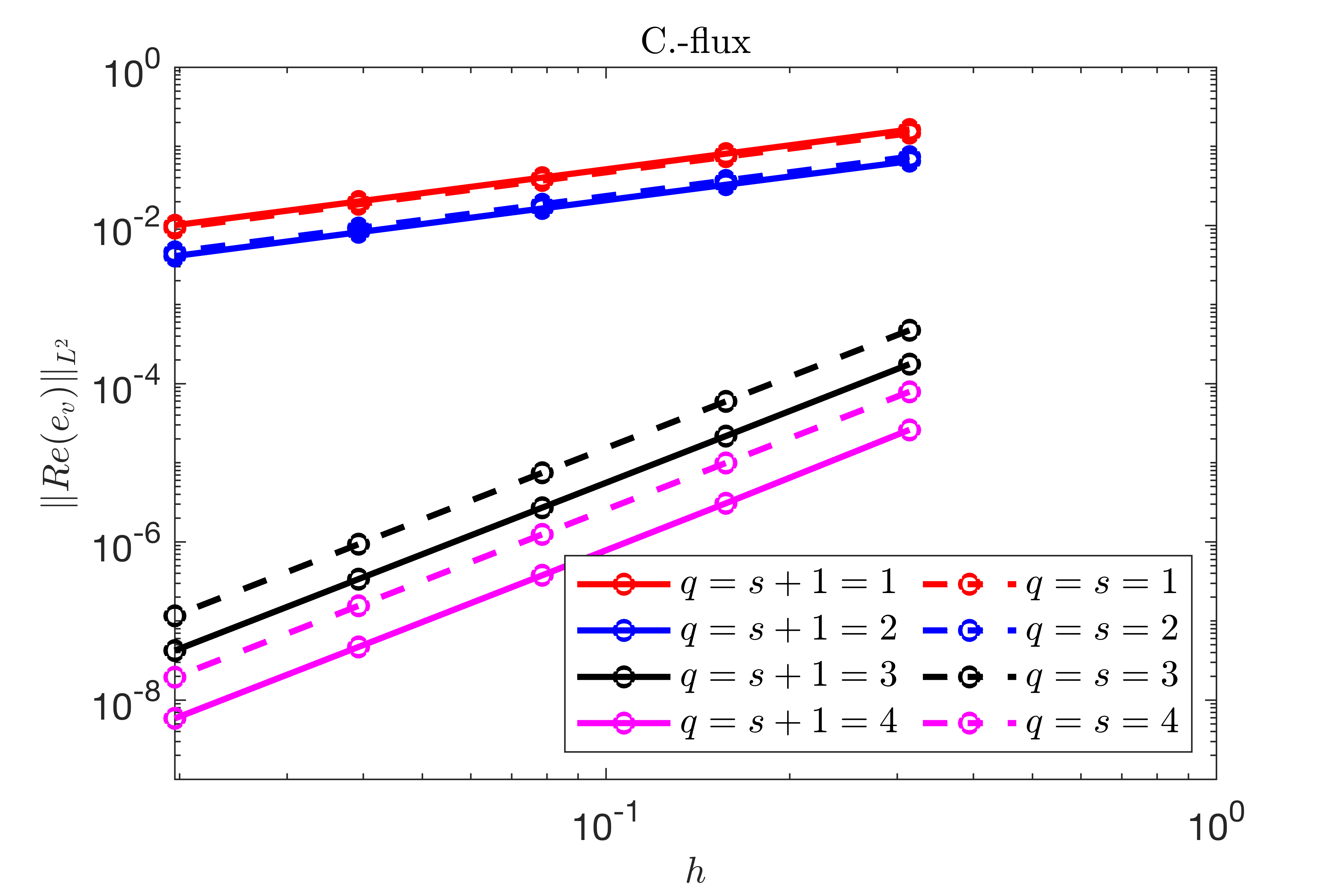}}
    \caption{\scriptsize{$L^2$ errors for the real parts of $u$ (top) and $v$ (bottom) for the problem ~\eqref{eq:51}. The solid lines show the results of $u_h$ and $v_h$ in the different approximation space with $q = s+1$, while the dashed lines show the results of $u_h$ and $v_h$ in the same approximation space with $s = q$.}}
    \label{fig:example1}
\end{figure}

\begin{table}[!htb]
 	\footnotesize
 	\begin{center}
 		\scalebox{1.0}{
 			\begin{tabular}{c c c c c c}
 				\hline
 				~ & ~ & ~ &  convergence rates of $u/v$ & ~ &~\\
 				\hline
 				flux/$(q,s)$ & ~ & $q=1, s=0$ & $q=2, s=1$ & $q=3, s=2$ & $q=4, s=3$ \\
 				\hline
 				S.-flux& ~ & 0.9588/0.9705 & 2.9708/1.9826 & 4.0075/2.9906 & \quad \quad 5.0128/3.9938 \\
                    A.-flux& ~ & 2.0004/1.0003 & 2.9638/2.0341 & 4.2050/2.9702 & \quad \quad 4.9061/4.0386 \\
 				C.-flux& ~ & 2.0101/1.0043 & 1.9991/0.9966 & 3.9890/3.0059 & \quad \quad 3.9853/3.0273 \\ 
 				\hline
     \hline
 				flux/$(q,s)$ & ~ & $q=s=1$ & $q=s=2$ & $q=s=3$ & $q=s=4$ \\
 				\hline
 				S.-flux& ~ & 1.0810/0.9625 & 2.9651/1.9716 & 4.0242/2.9829 & \quad \quad 5.0114/3.9873 \\
                    A.-flux& ~ & 0.2746/-0.0010 & 2.0053/1.0008 & 3.9977/2.0001 & \quad \quad 4.9997/2.9999 \\
 				C.-flux& ~ & 2.0104/0.9968 & 2.0000/1.0000 & 4.0359/2.9974 & \quad \quad 4.9942/2.9977 \\
 				\hline
 			\end{tabular}
 		}
 	\end{center}
    \vspace{-0.4cm}
 	\caption{\scriptsize{Linear least squares estimates of the rates of convergence from the curves in Figure \ref{fig:example1} for the real parts of $u$ and $v$ of the problem ~\eqref{eq:51}.}}\label{table_sq_diff_example1}
 \end{table}

The $L^2$ errors for the real parts of $u$ and $v$ at the final time $T = 1$ are shown in Figure \ref{fig:example1}, and the corresponding linear least squares estimates of the rates of convergence from the error curves can be found in Table \ref{table_sq_diff_example1}. Due to space limitations, we have omitted the results for the imaginary parts of $u,v$, since the performance is similar to the real parts in Figure \ref{fig:example1}. A clear conclusion is that the choice of $s = q - 1$ is more reliable for all except the C.-flux. And the details are as follows: 

\textbf{S.-flux:} For $u$, the $L^2$ error converges at the optimal rate, $q + 1$, when $q\geq 2$ for both cases $s = q$ and $s = q-1$; when $q = 1$, first-order convergence is observed for both $s = 0$ and $s = 1$. For $v$, the $L^2$ error converges at the optimal rate, $s + 1$, when $s\geq 0$ for the case with $s = q-1$, while it converges only at the suboptimal rate $s$ when $s\geq 1$ for the case with $s = q$.

\textbf{A.-flux:} For $u$, the $L^2$ error converges at the optimal rate, $q + 1$, when $q\geq 1$ for the case with $s = q-1$; when $s = q$, we observe an optimal convergence rate, $q+1$, when $q \geq 3$, while a second-order convergence when $q = 2$, and a zero-order convergence when $q = 1$. For $v$, the $L^2$ error converges at the optimal rate, $s + 1$, when $s\geq 0$ for the case $s = q-1$, and a reduction of $2$ compared to the optimal convergence, $s+1$, when $s = q$.

\textbf{C.-flux:} For $u (\mbox{resp.} \ v)$ with the case $s = q-1$, the $L^2$ error converges at the optimal rate, $q + 1 (\mbox{resp.}\ s+1)$, when $q$ is odd, while at a suboptimal rate, $q (\mbox{resp.}\ s)$, when $q$ is even. For the case $q = s$, the $L^2$ error of $u$ converges at the optimal rate, $q + 1$, when $q = 1,3,4$, and at a suboptimal rate, $q$, when $q = 2$; while the $L^2$ error of $v$ converges at the suboptimal rate, $s$, when $s$ is odd, and at a $(s-1)$-th convergent rate when $s$ is even. 

 Figure \ref{err_history} shows the time history of the $L^2$ errors in the modulus of $u$ for three different numerical fluxes. We focus on the scenario where $u^h$ and $v^h$ belong to different approximation spaces with $s = q-1 = 2$. Specifically, in Figure \ref{err_history}, we plot the $L^2$ errors in the modulus of $u$ with $(q, N) = (2,80)$ up to the final time $T = 400$ from left to right are for the S.-flux, the C.-flux, and the A.-flux, respectively. It is worth noting that the $L^2$ errors grow at most linearly with time for all three numerical fluxes, although Theorem \ref{truth2} suggests an exponential growth rate with time.

\begin{figure}[!htb]
\centering
{\includegraphics[width=0.32\textwidth]{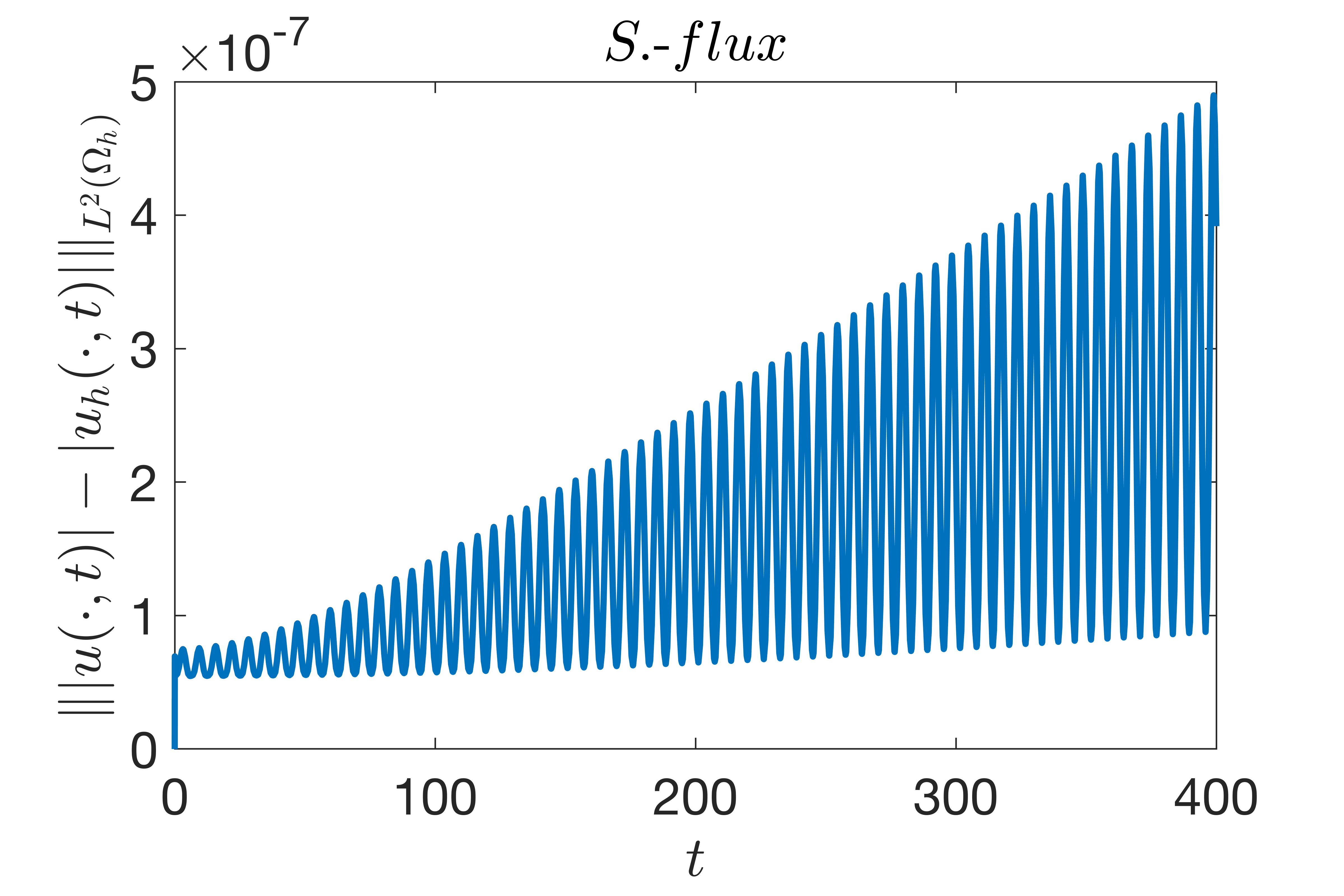}}
{\includegraphics[width=0.32\textwidth]{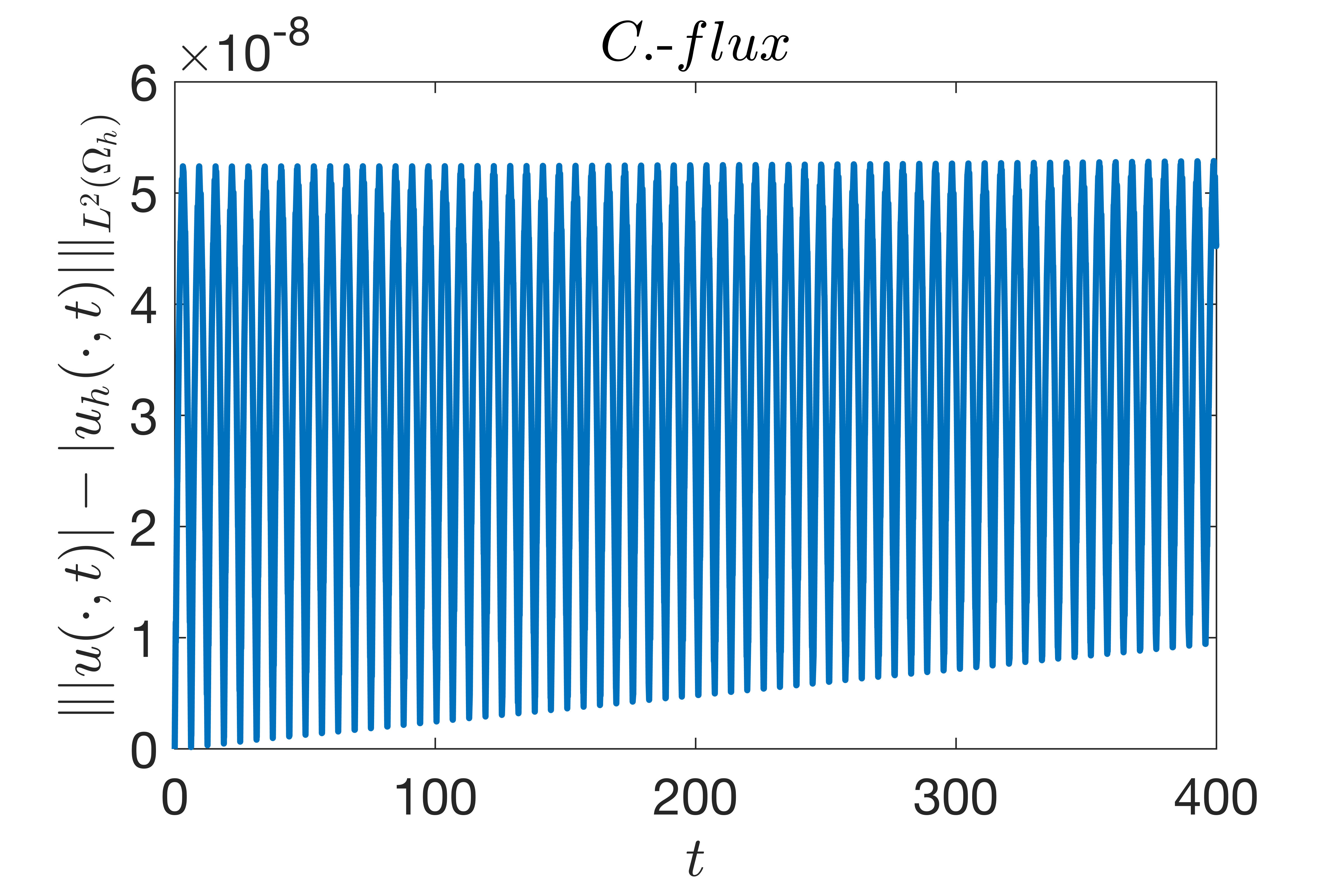}} 
{\includegraphics[width=0.32\textwidth]{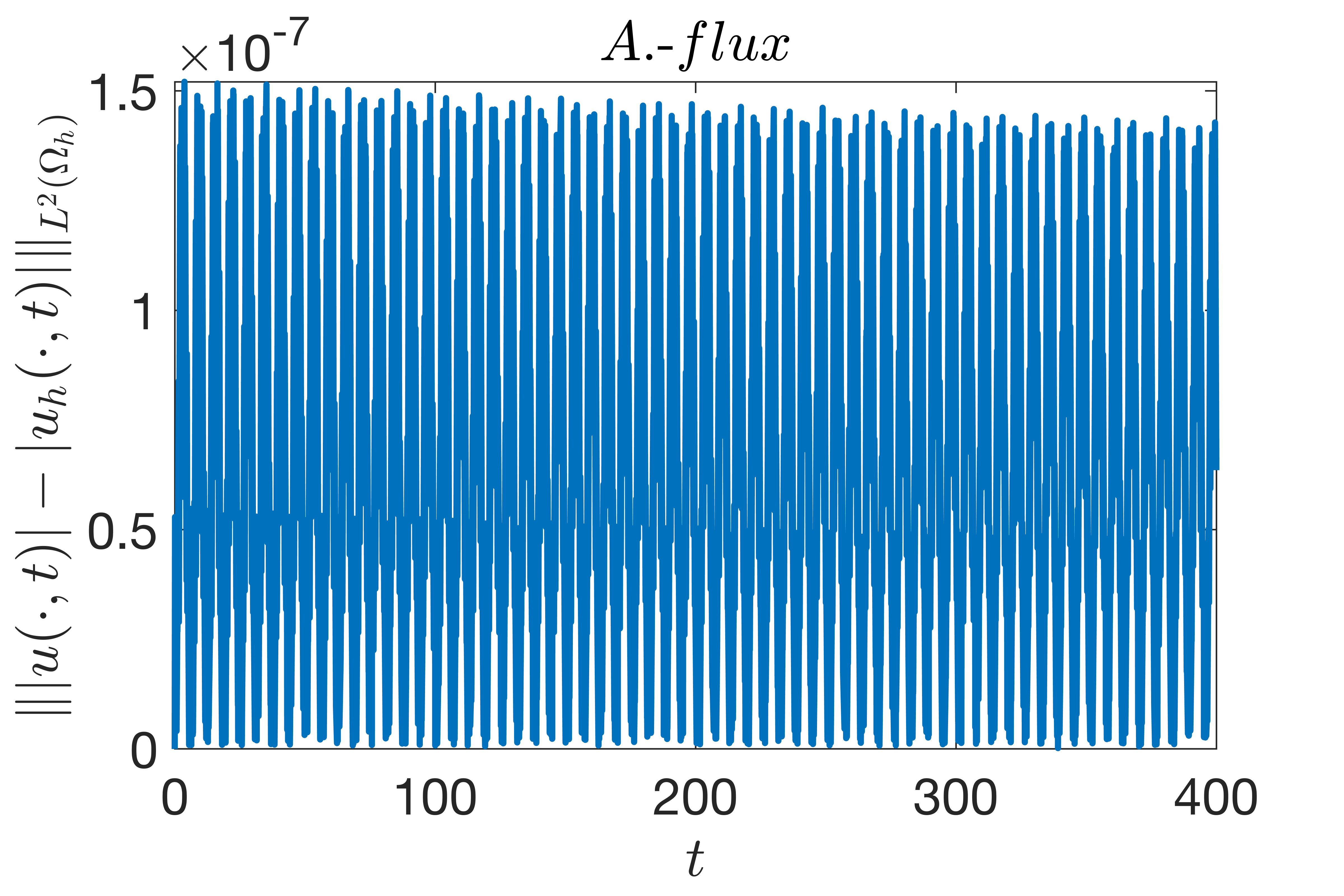}} 
    \caption{\scriptsize{From left to right is the time history of $L^2$ errors in $|u|$ for the problem ~\eqref{eq:51} using the S.-flux, the C.-flux, and the A.-flux, respectively, up to the final time $T = 400$.  In addition, the DG solutions $(u^h, v^h)$ are set to be in different approximation spaces with $s = q-1 = 2$, and the number of elements is chosen to be $N = 80$.}}
    \label{err_history}
\end{figure}

\subsubsection{Example II}

For the second example, we consider the case where $\beta(x) = e^{-x^2}$ is spatially dependent, while we still have $\alpha = 1$ and $f(|u|^2) = 1$ as in the first example, namely,
\begin{align}\label{eq:52}
\begin{split}
    & u_{tt} - u_{xx} + i u_t + e^{-x^2} u = 0, \quad x \in (0, 2\pi)\,,\\
    & u(x,0) = e^{ix}, \quad u_t (x,0) = i e^{ix}\,,
\end{split}
\end{align}
with the periodic boundary condition $u(0,t)=u(2\pi,t)$.

We use the same spatial discretization as in the previous example. Figure \ref{beta_nonlinear} shows time snapshots of $|u|$ at $t = 5\pi^2, 10\pi^2$, and $20\pi^2$. The S.-flux is used with $u^h$ and $v^h$ in the same approximation space with $s = q = 3$ and the uniform mesh with the number of elements $N = 160$. The only difference between problem ~\eqref{eq:51} and problem ~\eqref{eq:52} is the value of $\beta(x)$. In problem ~\eqref{eq:51}, $|u|$ is equal to $1$, as seen in ~\eqref{sol:eq51}, while this is not the case in problem ~\eqref{eq:52}, as shown in Figure \ref{beta_nonlinear}.
\begin{figure}[!htb]
    \centering
{\includegraphics[width=0.32\textwidth]{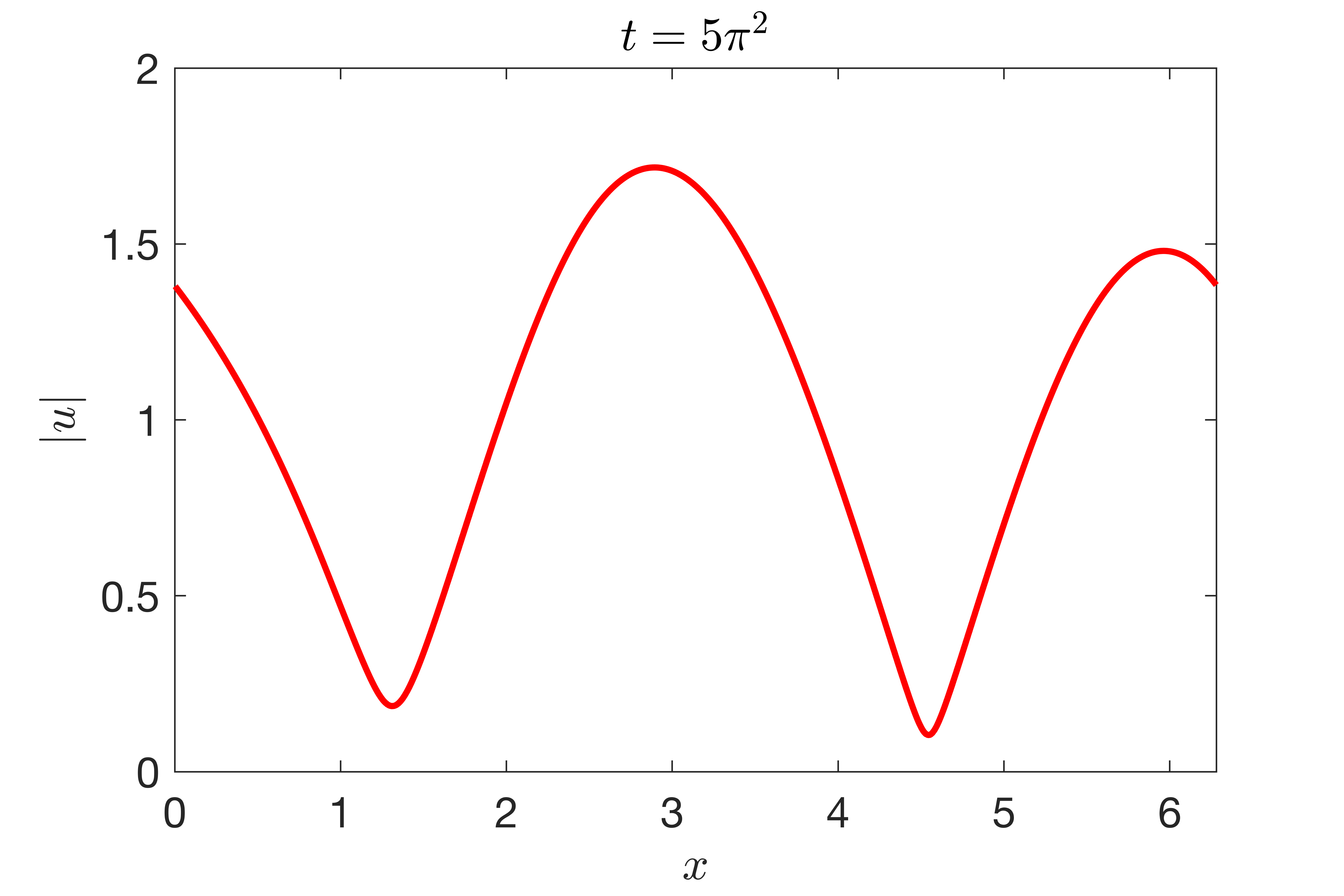}} 
{\includegraphics[width=0.32\textwidth]{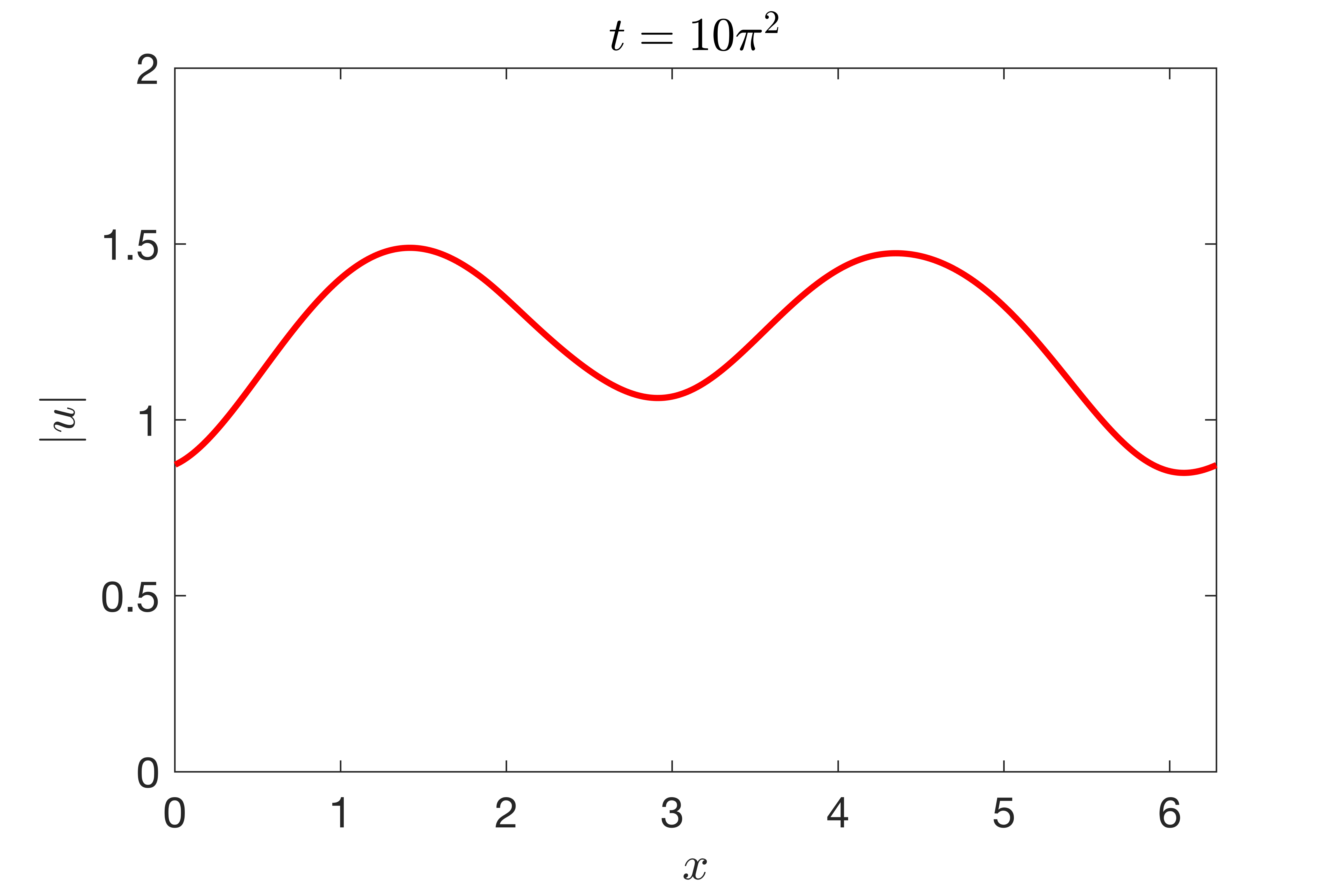}}
{\includegraphics[width=0.32\textwidth]{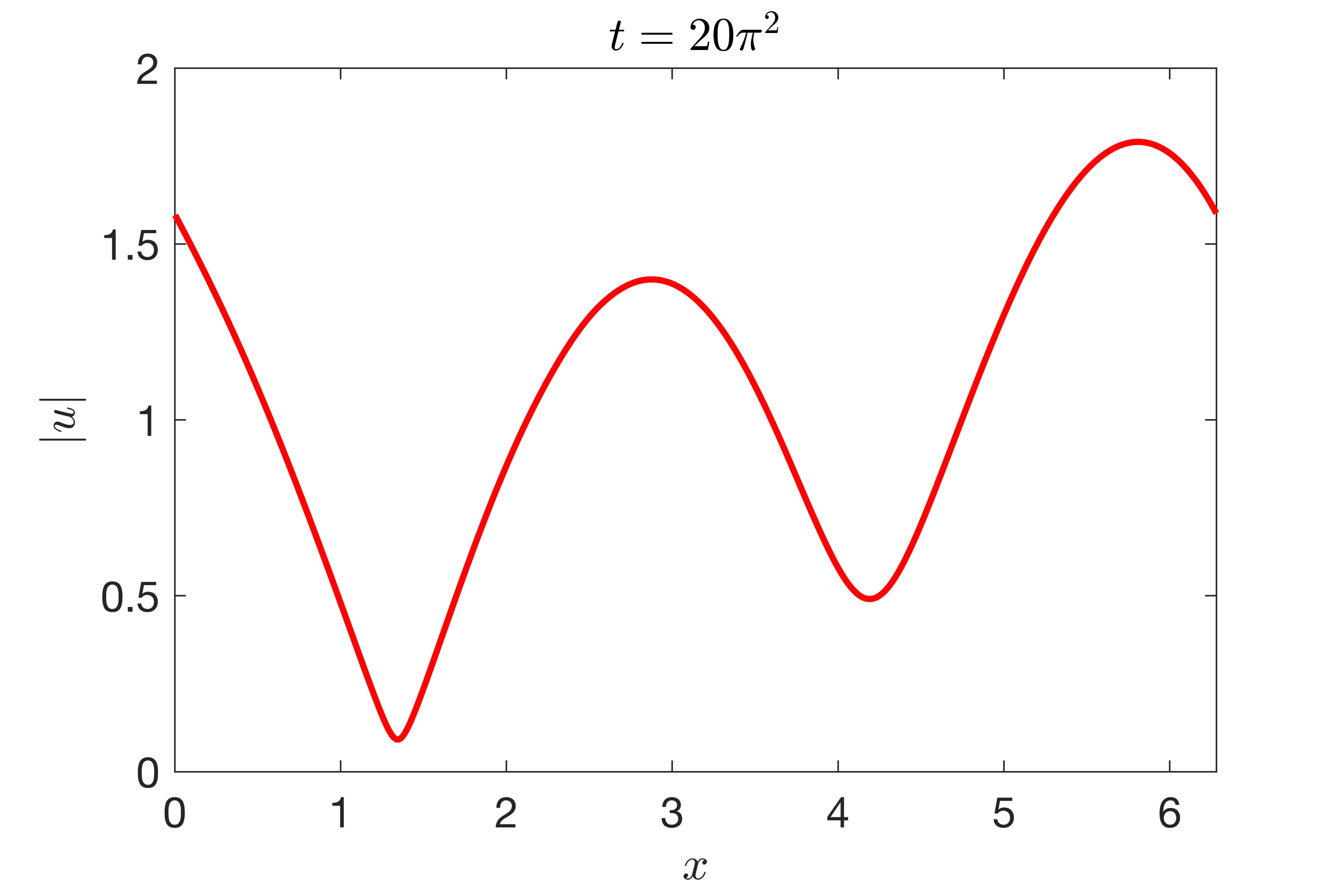}}
    \caption{\scriptsize{The time evolution of $|u|$ for the problem ~\eqref{eq:52} using the S.-flux on a uniform mesh of $N = 160$ at different times $t=5\pi^2, 10\pi^2, 20\pi^2$. In particular, $u^h$ and $v^h$ are in the same approximation space with $s = q = 3$.}}
    \label{beta_nonlinear}
\end{figure}

In addition, we present the discrete numerical energy trajectories $E^h(t) = \sum_j E_j^h(t)$ for the proposed EDG scheme ~\eqref{eq:21}--\eqref{eq:22} applied to the problem ~\eqref{eq:52}. Figure \ref{energy_eg2} shows the results obtained with three different numerical fluxes: the A.-flux, the C.-flux, and the S.-flux. Here $E_j^h(t)$ is defined in ~\eqref{dis_energy}. We also consider two cases: $u^h$ and $v^h$ are in the same approximation space with $s = q = 3$ (bottom), and $u^h$ and $v^h$ are in different approximation spaces with $s = q-1 = 2$ (top). The number of elements is chosen to be $N = 160$, with the final time $T = 100$. Overall, we observe excellent energy conservation for the conservative schemes, i.e. the A.-flux and the C.-flux for both cases with either $s = q-1$ or $s = q$. In particular, the discrete energy is conserved around $8$ digits. However, for the energy dissipating scheme, i.e. the S.-flux, the discrete energy is dissipated as expected, but the total dissipation is small, only about $6$ digits up to $T = 100$, by having a slightly different performance depending on $s$.

\begin{figure}[!htb]
\centering
{\includegraphics[width=0.4\textwidth]{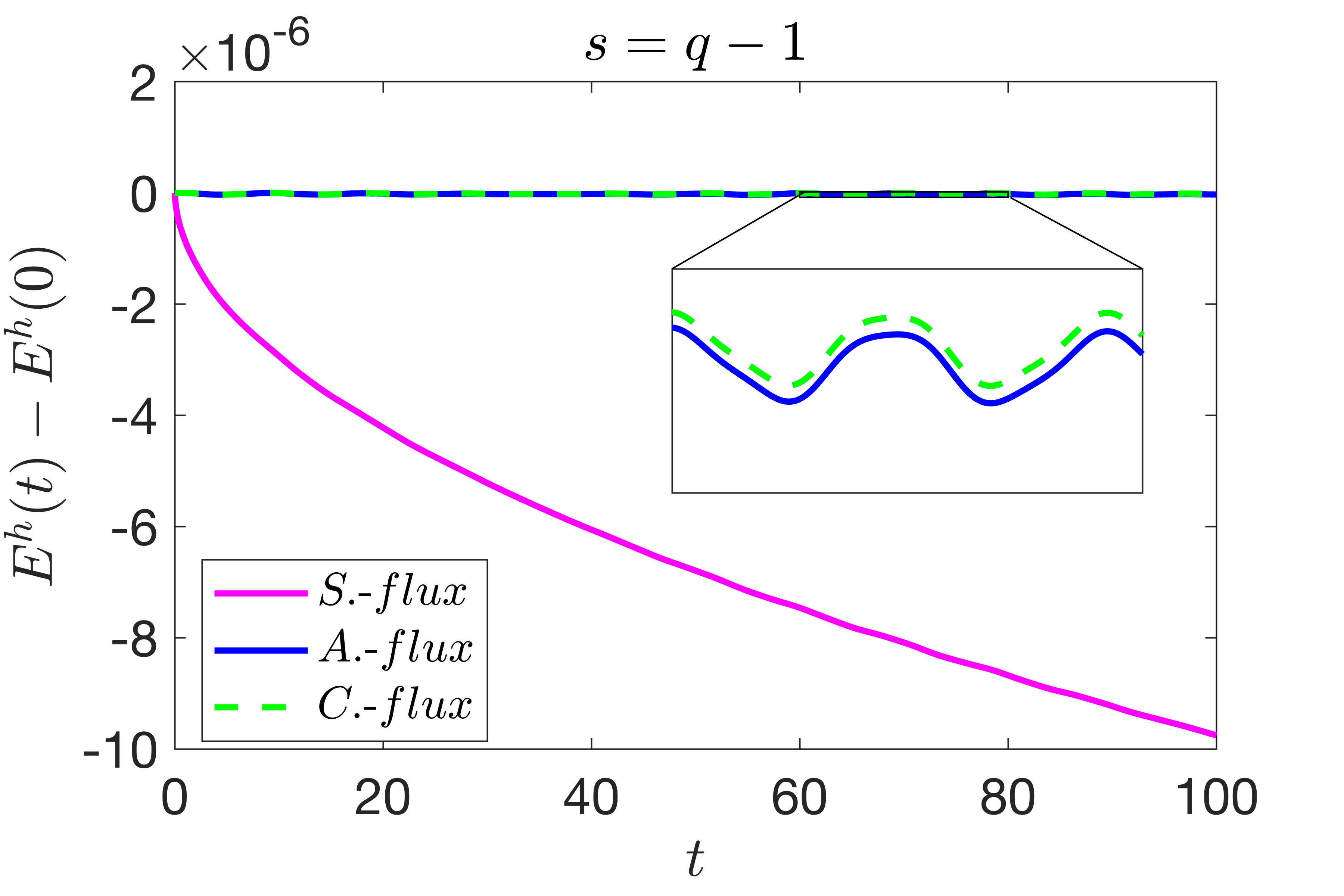}}\quad
{\includegraphics[width=0.4\textwidth]{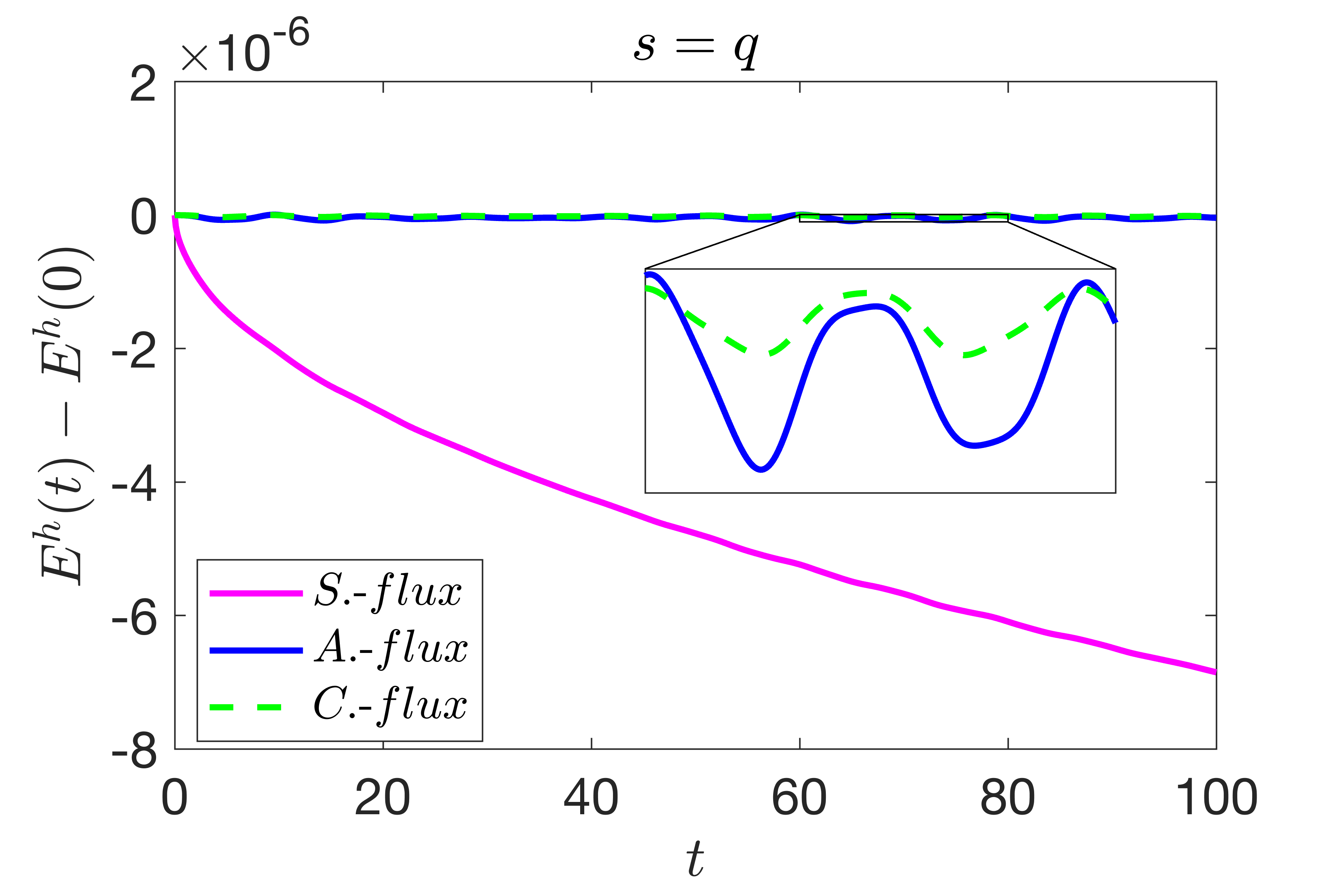}}
	\caption{\scriptsize{The time evolution of the discrete energy difference $E^h(t) - E^h(0)$ for the problem ~\eqref{eq:52}. The left panel is for the case where $u^h$ and $v^h$ are in different approximation spaces with $s=q-1=2$, while the right panel is for the case where $u^h$ and $v^h$ are in the same approximation space with $s=q=3$.
	}}\label{energy_eg2}
\end{figure}

\subsubsection{Example III}

We now start to test the nonlinear models. We start with the case where $   \alpha = \beta = 1$ and $ f(|u|^2) = |u|^2$. In this case, model~\eqref{EQ:NLSW} simplifies into the form
\begin{align}\label{eq:53}
\begin{split}
    &u_{tt} - u_{xx} + i u_t + |u|^2 u = 0, \quad x \in (-40, 40)\,,\\
    & u(x,0) = (1+i) x e^{-10(1-x)^2}, \quad u_t (x,0) = 0.
\end{split}
\end{align}
Unlike the previous two examples, we impose homogeneous Dirichlet boundary conditions $u(-40, t) = u(40, t) = 0$ instead of periodic boundary conditions. 

\begin{figure}[!htb]
\centering
{\includegraphics[width=0.32\textwidth]{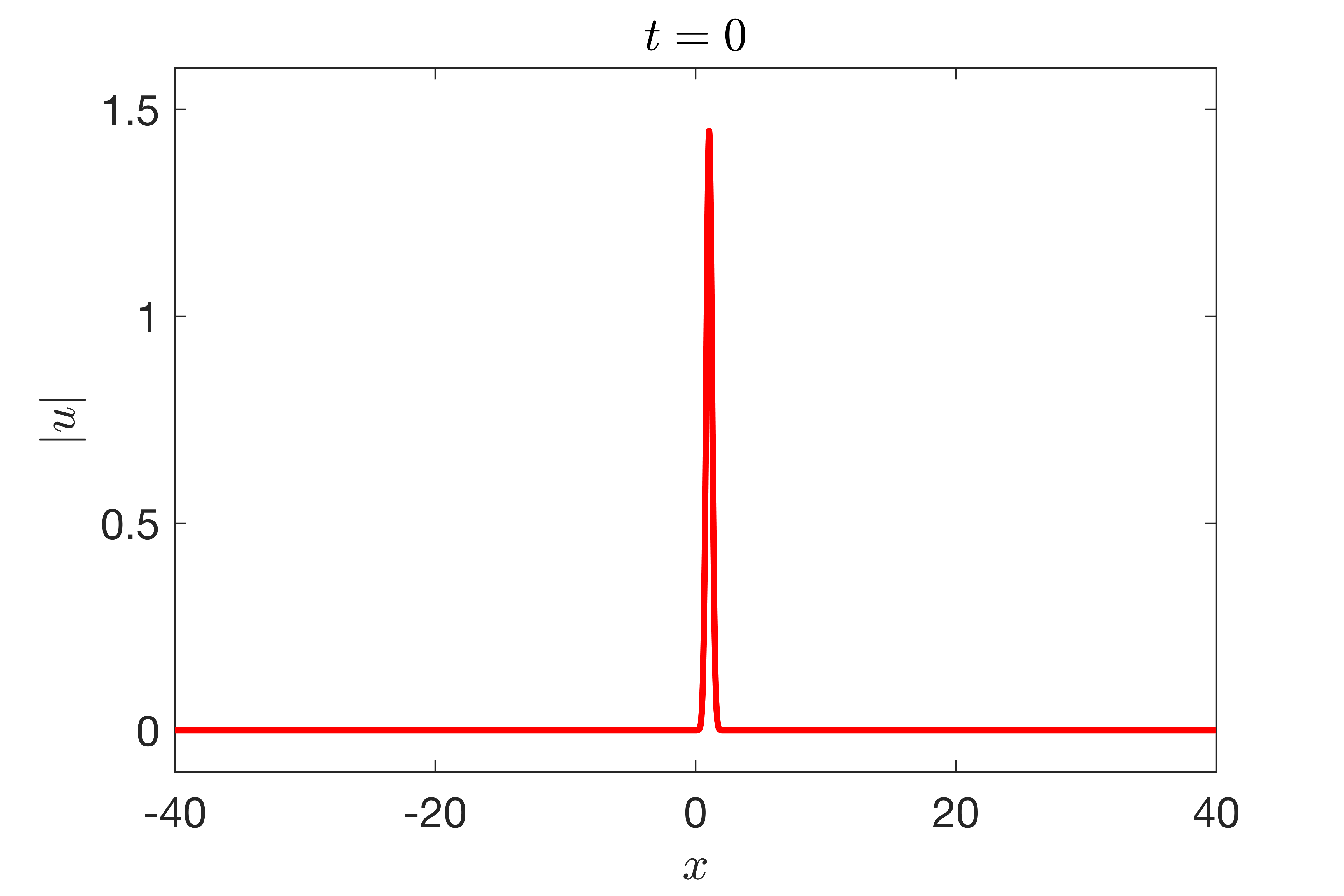}} 
{\includegraphics[width=0.32\textwidth]{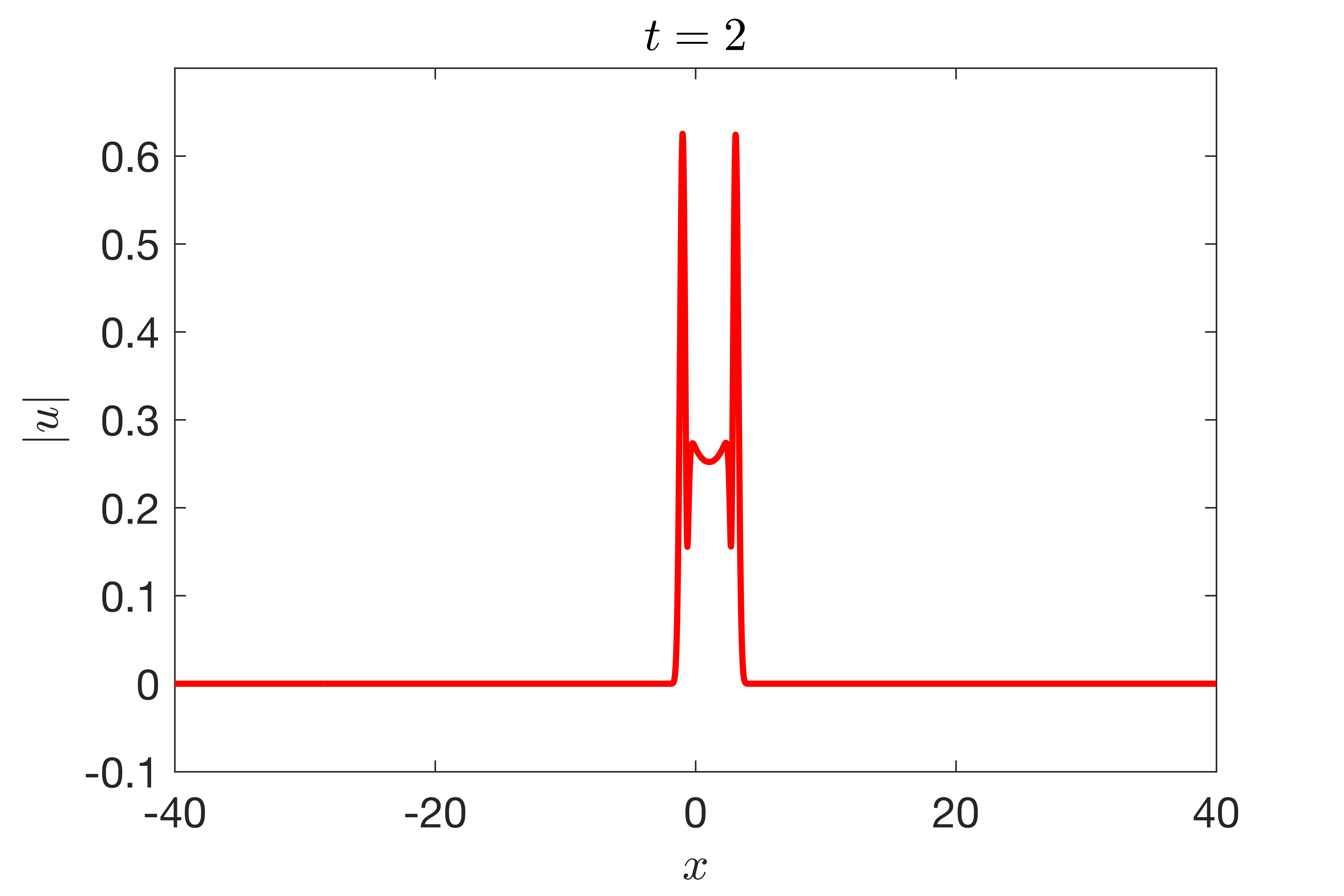}} 
{\includegraphics[width=0.32\textwidth]{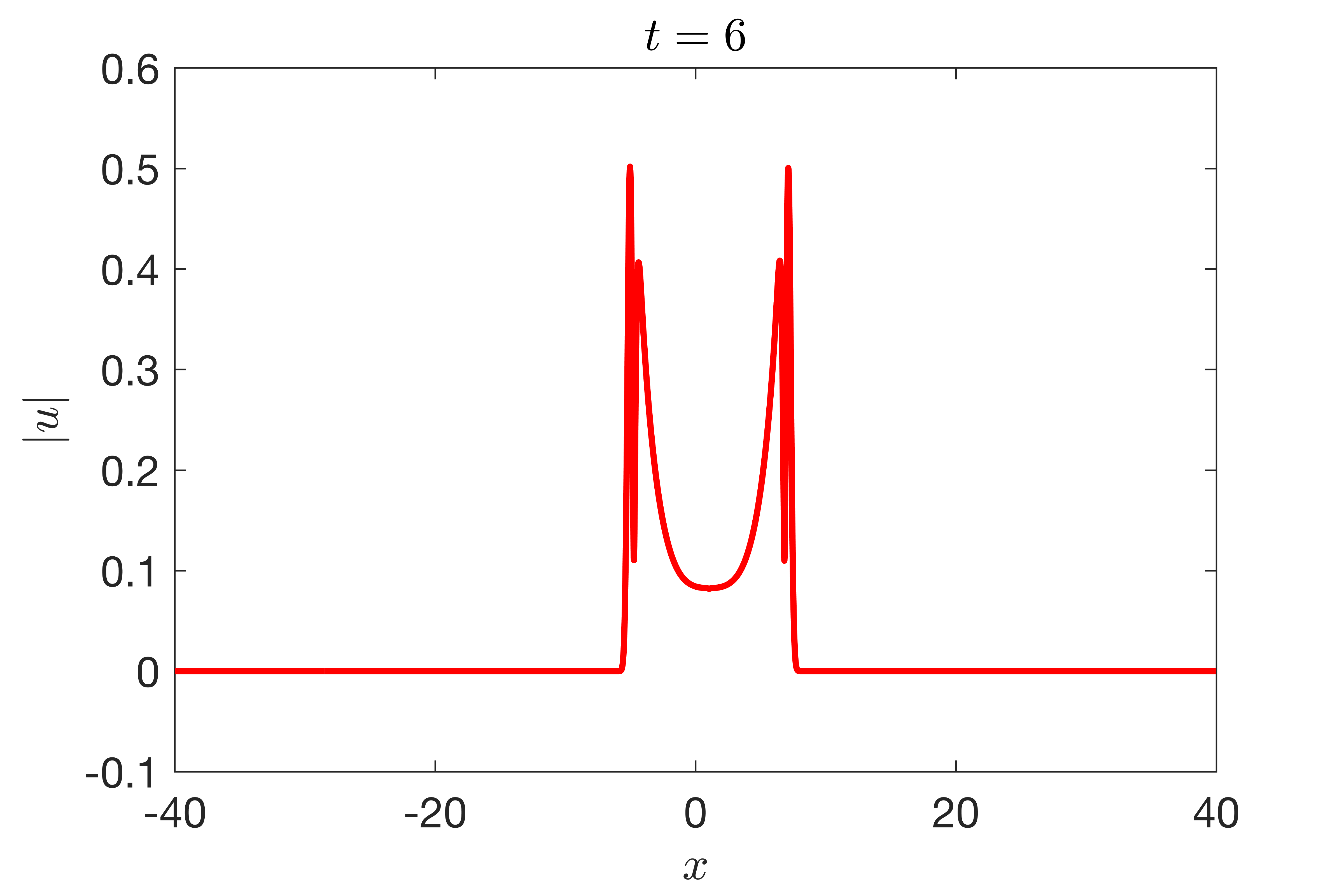}}\\
{\includegraphics[width=0.32\textwidth]{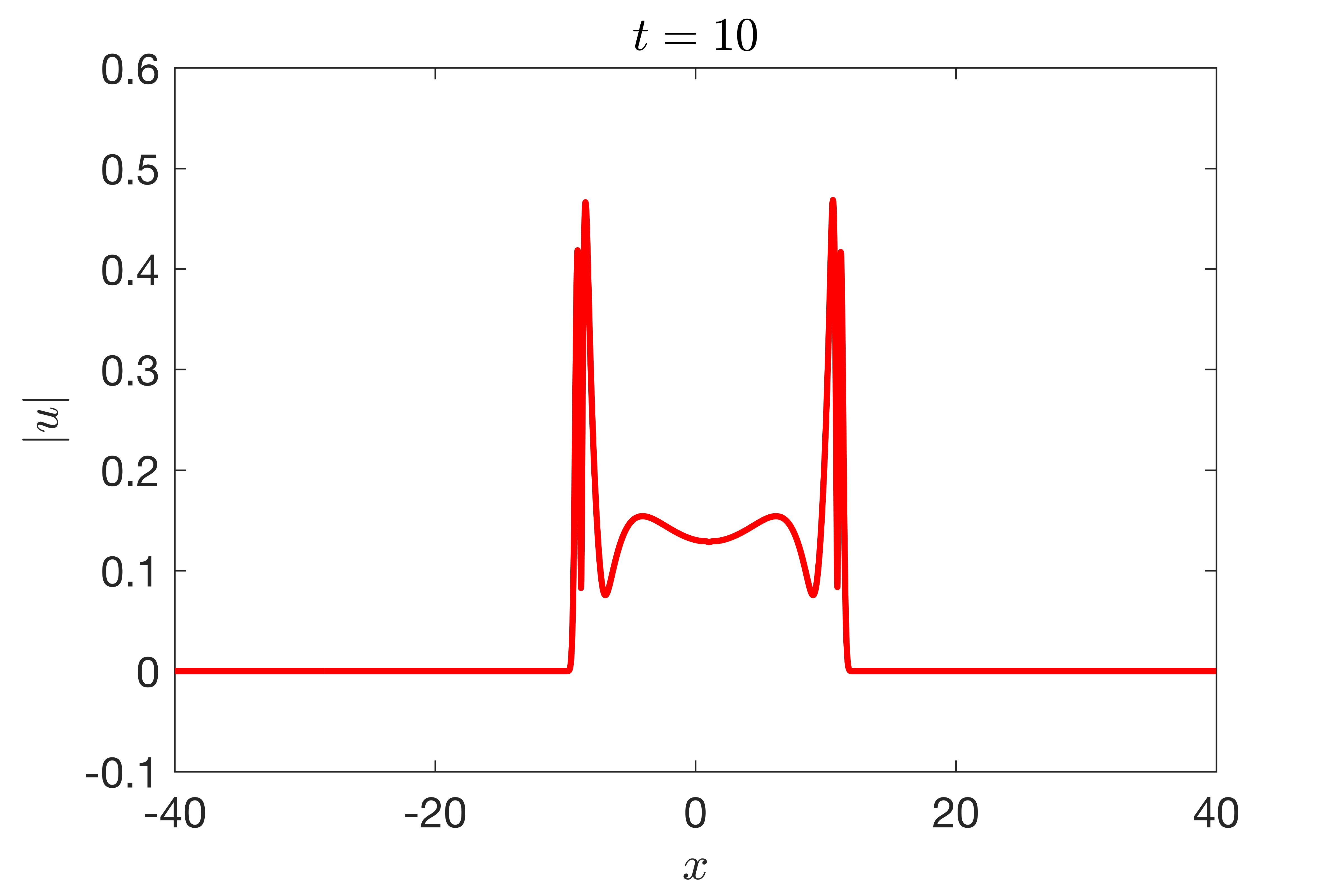}} 
{\includegraphics[width=0.32\textwidth]{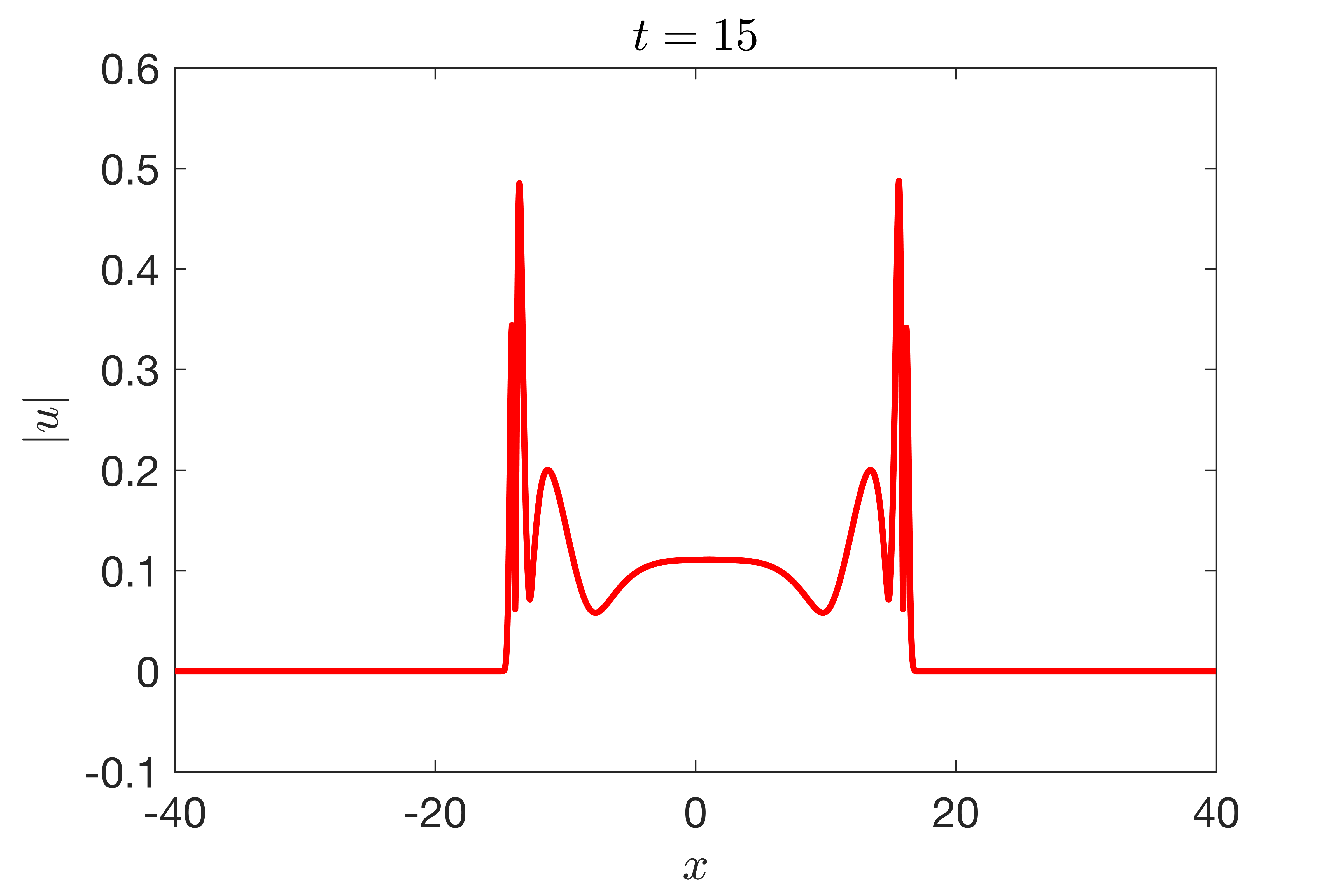}}
{\includegraphics[width=0.32\textwidth]{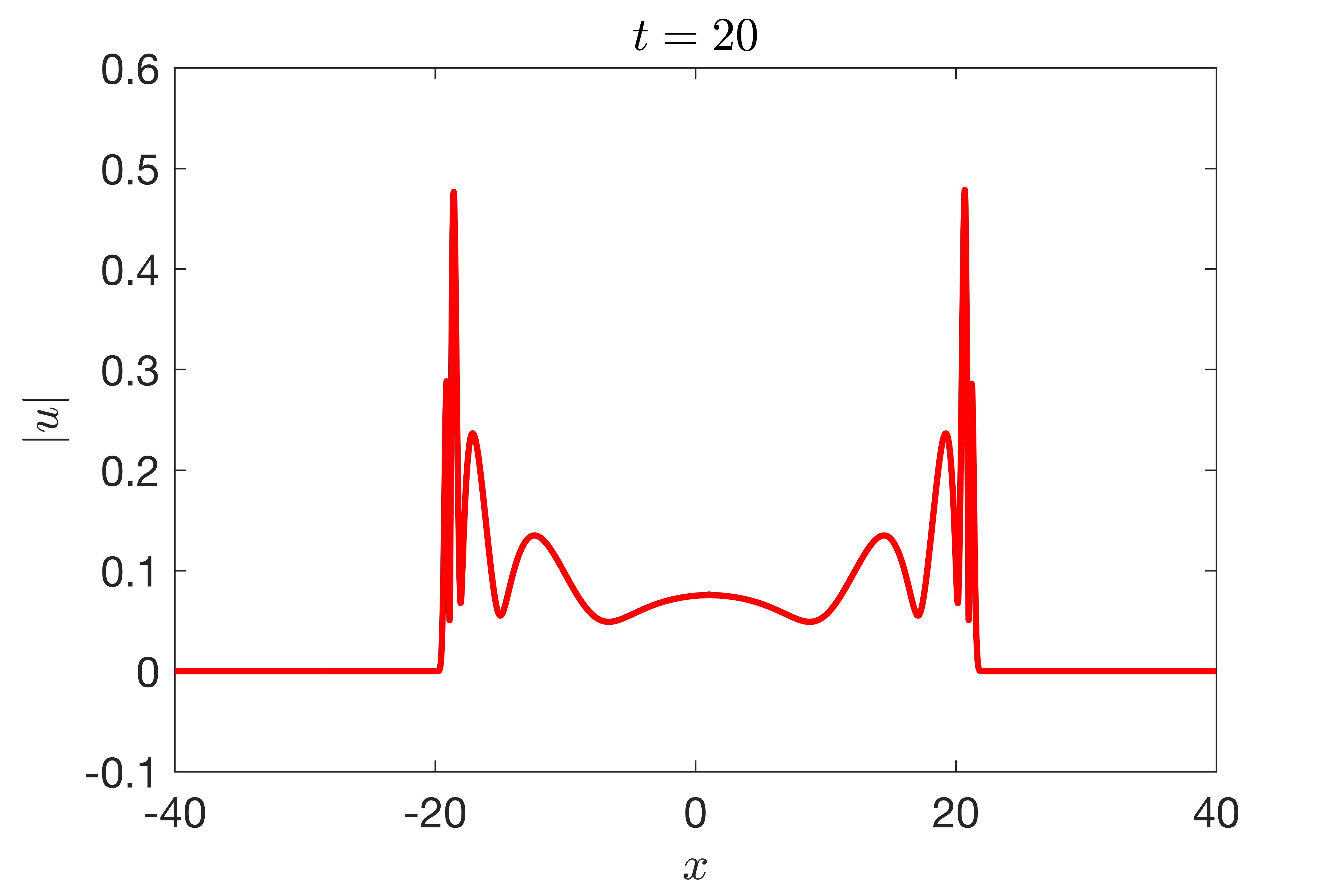}}
    \caption{\scriptsize{The motion of the soliton $|u|$ for the problem ~\eqref{eq:53} using the S.-flux on a uniform mesh of $N = 1600$ at different times $t=0,2,6,10,15,20$. In particular, $u^h$ and $v^h$ are chosen to be in the same approximation space with $s = q = 2$.}}\label{fig:solitonu}
\end{figure}
To solve the problem~\eqref{eq:53}, we use the same spatial discretization as in Section~\ref{liner_sec}. In particular, we choose the number of elements to be $N = 1600$, and $u^h$ and $v^h$ to be in the same approximation space with $s = q = 2$. In addition, the S.-flux is used in the simulation. The time evolution of the modulus of $u$ at different instances ($t=0,2,6,10,15,20$) is shown in Figure \ref{fig:solitonu}. We observe that our proposed scheme ~\eqref{eq:21}--~\eqref{eq:22} is comparable to the one in \cite{li2015ldg}: the scheme is stable, and there is a blow-up phenomenon.

\subsubsection{Example IV}

In the final numerical example in one-dimensional space, we consider model~\eqref{EQ:NLSW} with parameters $\alpha = 1$, $\beta = -2$ and $f(|u|^2) = |u|^2$. The specific form of the equation is
\begin{equation}
    \label{eq:54}
    u_{tt} - u_{xx} + i u_t - 2 |u|^2 u = 0, \quad x \in (-50, 50)\,.
\end{equation}
It turns out that there is an exact solution to this equation~\cite{Wang-JCM07} in the form
\begin{equation}\label{sol:eq54}
    u(x,t) = A \mbox{sech}(Jx) e^{i \Theta t},
\end{equation}
where $A = |J|, \quad \Theta = \frac{1}{2} (-1 \pm \sqrt{1 - 4 J^2})\,.$
Here, we take $J = \frac{1}{4}$ and $\Theta = -\frac{1}{2} - \frac{\sqrt{3}}{4}$ for our simulations. This means that the initial conditions are $u(x,0)=A\sech(Jx)$ and $u_t(x,0)=i\Theta A\sech(Jx)$. The Dirichlet boundary conditions are given by evaluating ~\eqref{sol:eq54} at the boundary points of the domain.

Figure~\ref{nonlinear_exact} shows the modulus of $u$, denoted as $|u|$, at the final time $T = 2$, using polynomial spaces of different degrees. Specifically, we use the S.-flux, and $u^h$ and $v^h$ are chosen to be in the same approximation spaces with $s = q$ and a uniform mesh with $N = 400$. It is noticeable that the simulation results with higher-order polynomials outperform those with lower-order polynomials. This is because the higher-order methods have lower dispersion and dissipation errors than the lower-order methods. Therefore, the EDG method can effectively model certain problems by adjusting its accuracy level as needed since it can achieve arbitrary order.
\begin{figure}[!htb]
\centering
{\includegraphics[width=0.4\textwidth]{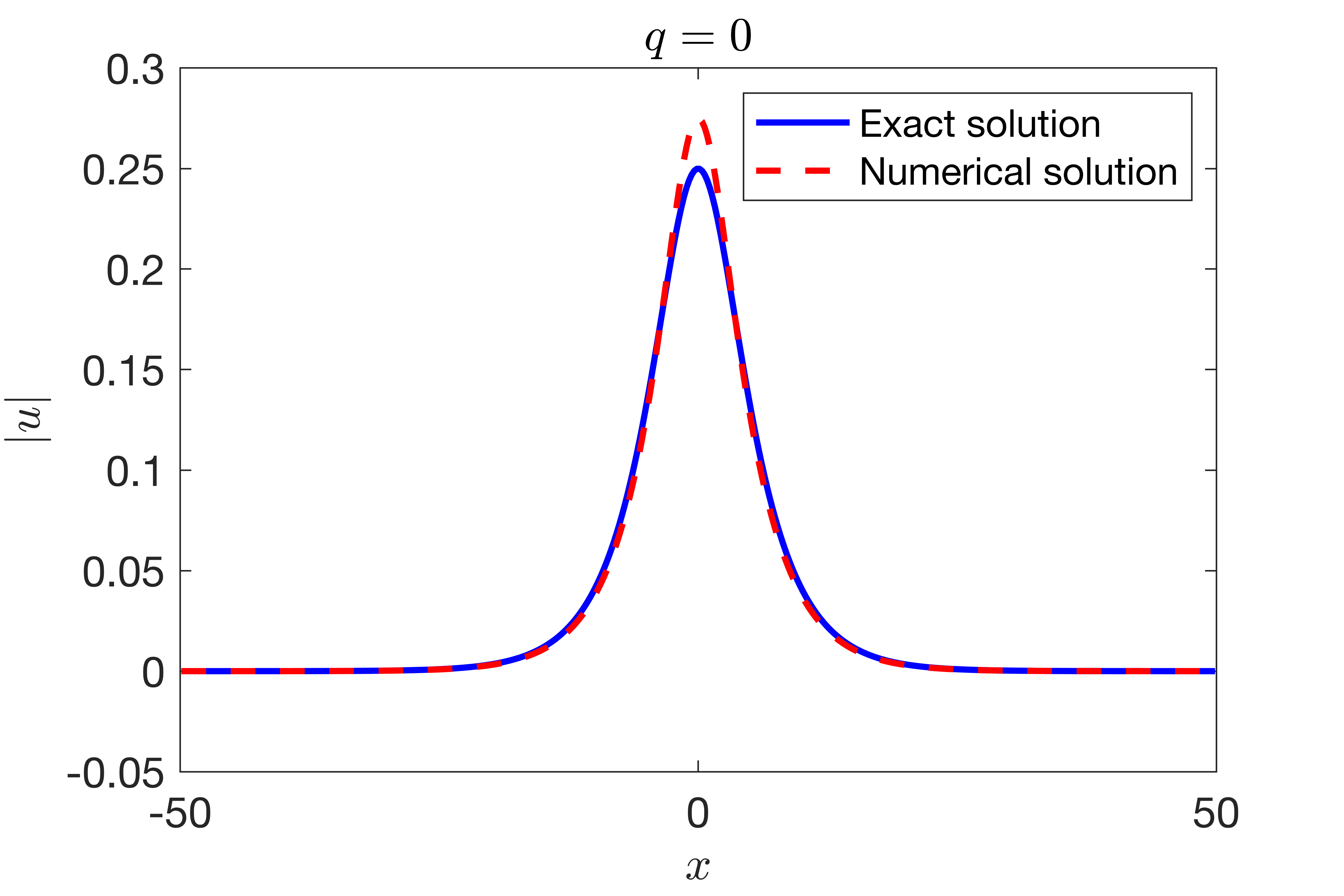}}
{\includegraphics[width=0.4\textwidth]{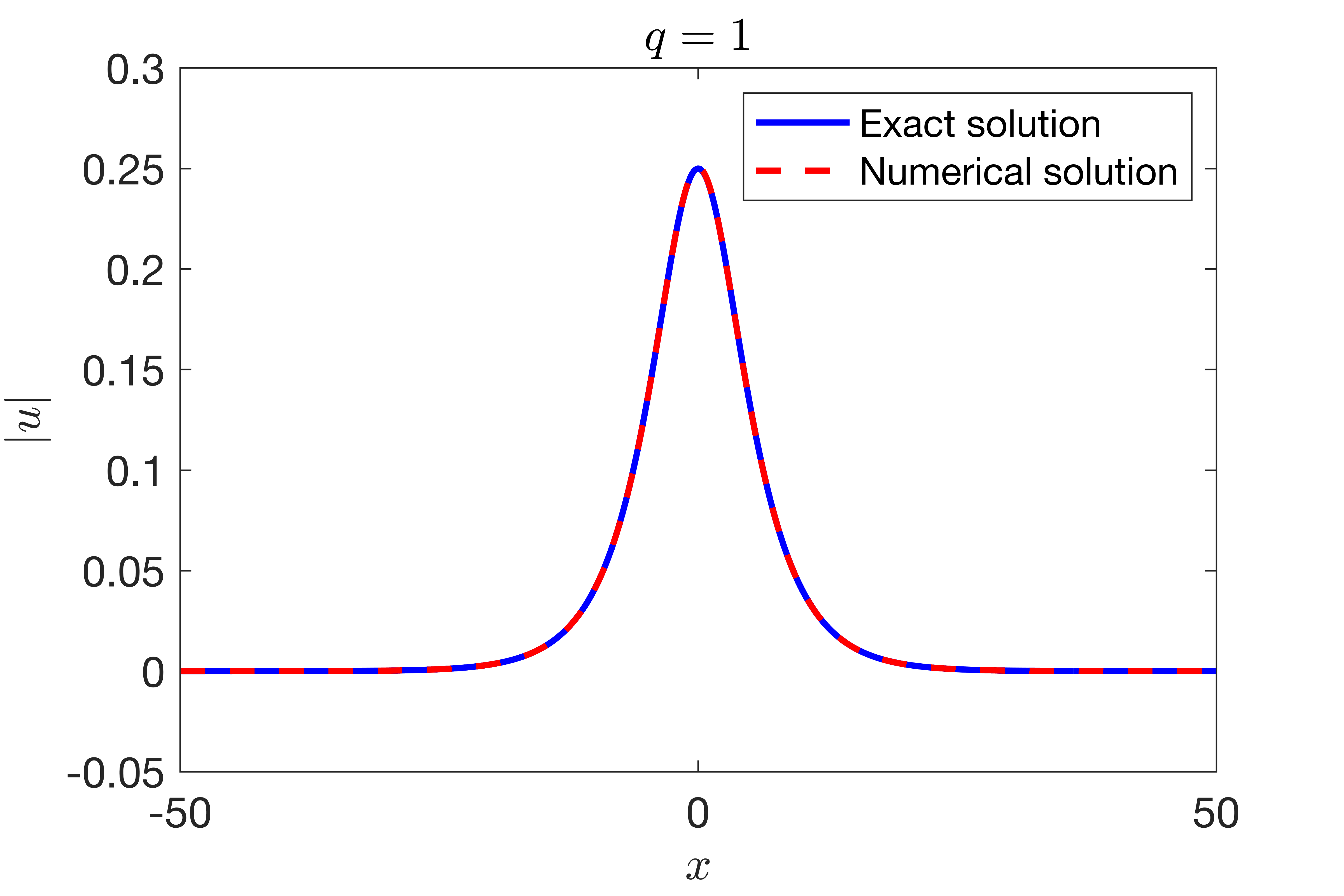}} \\
{\includegraphics[width=0.4\textwidth]{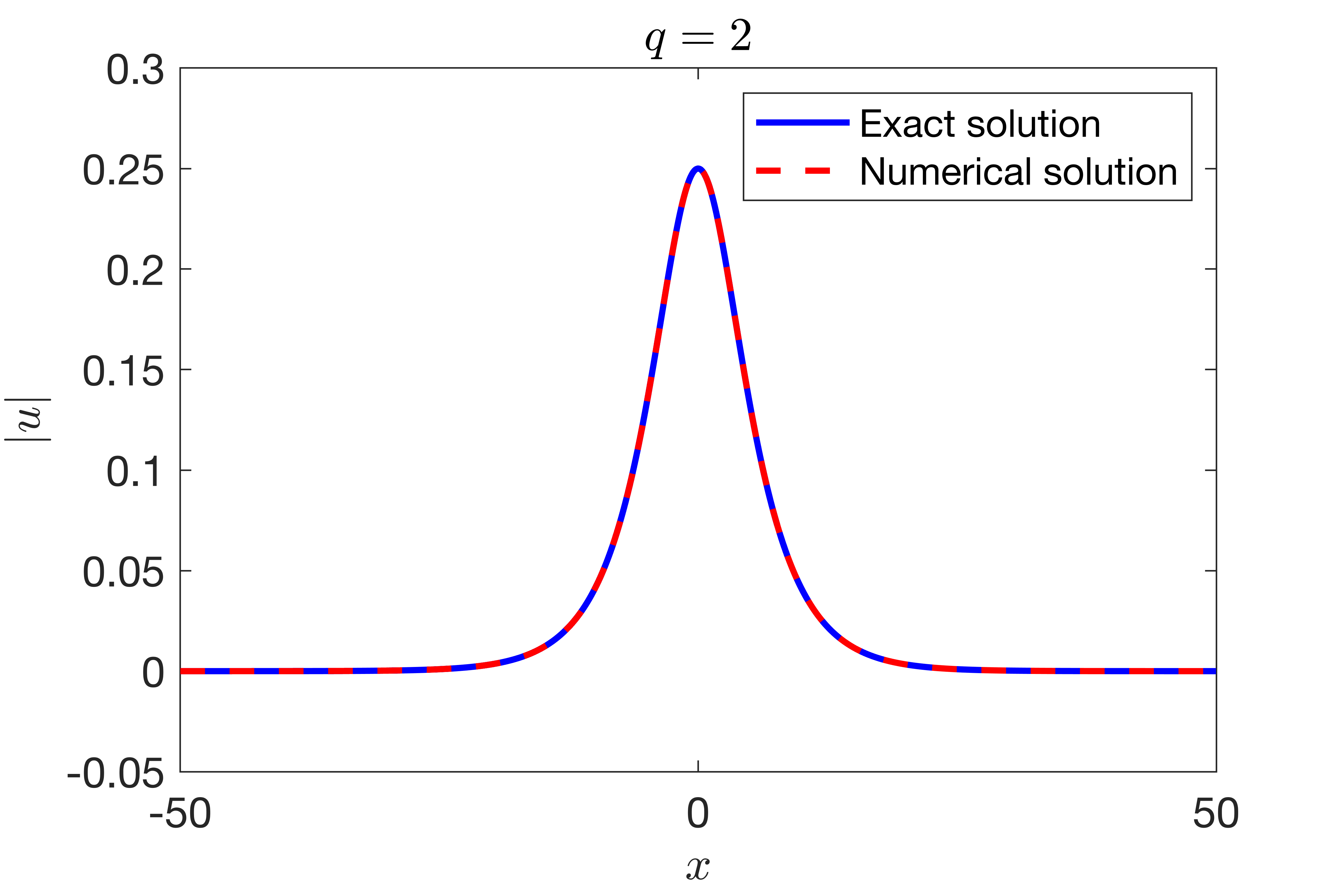}}
{\includegraphics[width=0.4\textwidth]{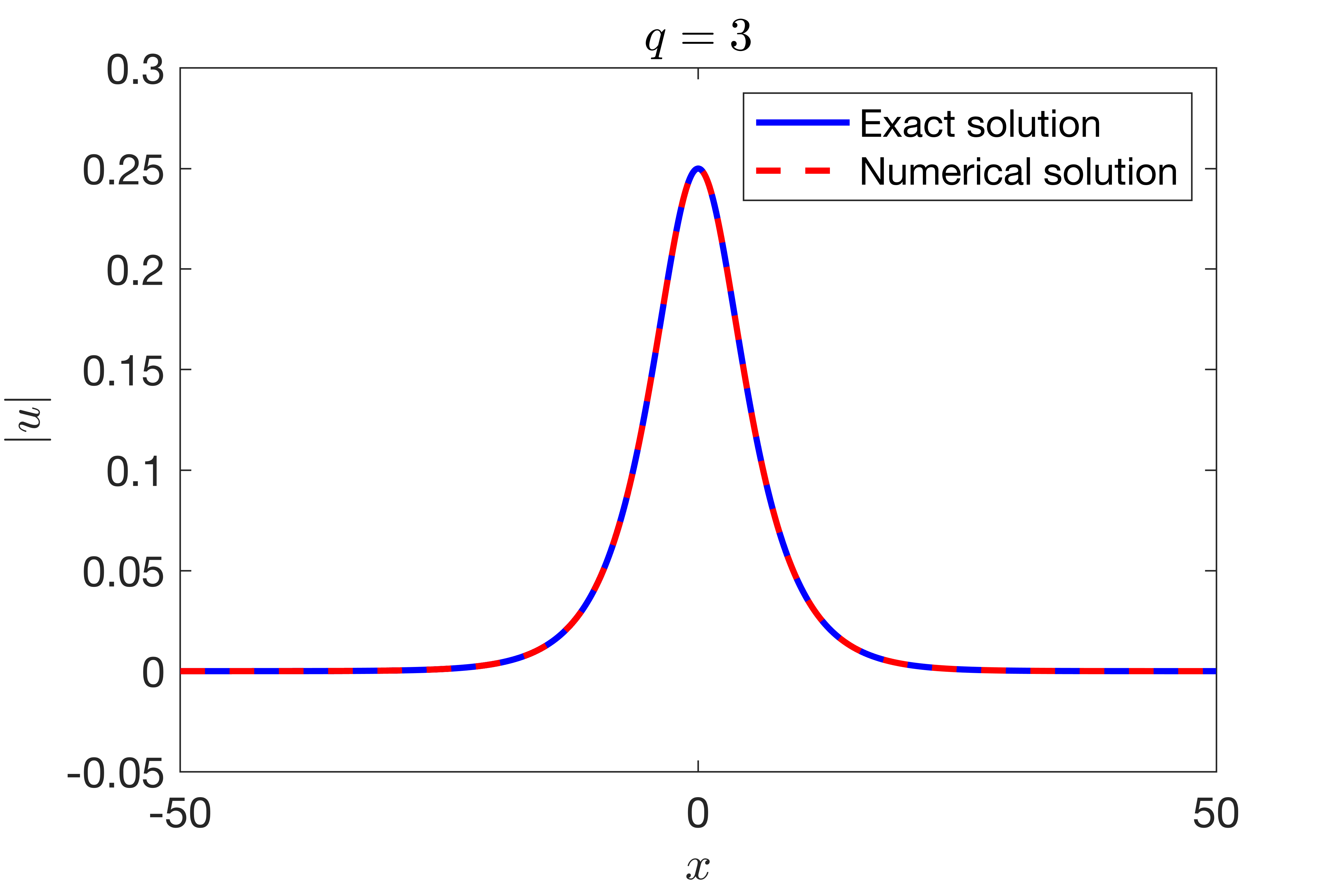}}
    \caption{\scriptsize{The snapshot of $|u|$ for problem (\ref{eq:54}) on a uniform mesh of $N = 400$ at the final time $T=2$ when using the S.-flux. Here we consider the case where $u^h$ and $v^h$ are in the same approximation space with $s = q$.}}
    \label{nonlinear_exact}
\end{figure}

We also present the $L^2$ errors for the real and imaginary parts of $u$ and $v$ in Table \ref{1Dnonlinear_sommer_uv} with the S.-flux at the final time $T = 2$. Here we consider the case where $u^h$ and $v^h$ are in the different approximation spaces with $s = q-1$. As with the linear problem ~\eqref{eq:51}, we observe an optimal rate of convergence for $u$ and $v$ when $q \geq 2$; and first-order convergence for both $u$ and $v$ when $q = 1$.
\begin{table}[!htb]
 	\footnotesize
 	\begin{center}
 		\scalebox{1.0}{
 			\begin{tabular}{c c c c c c c c c c}
 				\hline
 				~ & ~ &  $\mbox{Re}(u)$ &~ & $\mbox{Re}(v)$ &~& $\mbox{Im}(u)$& ~ & $\mbox{Im}(v)$ & ~\\
 				\cline{3-6} \cline{7-10}
 				$(q,s)$ & $N$ & $L^2$ error & order & $L^2$ error & order & $L^2$ error & order & $L^2$ error & order \\
 				\hline
 				(1,0)& 50 & 2.7632e-04 & --& 8.1075e-04 & -- & 1.1553e-03 & -- & 4.1189e-03 & --\\
 				~& 100 & 1.5310e-04 & 0.8518 & 4.5135e-04 & 0.8450 & 4.4889e-04 & 1.3639 & 2.0923e-03 & 0.9772\\
 				~& 200 & 8.0268e-05 & 0.9316 & 2.3736e-04 & 0.9272 & 2.1507e-04 & 1.0616 & 1.0542e-03 & 0.9890 \\
 				~& 400 & 4.1040e-05 & 0.9678 & 1.2156e-04 & 0.9654 & 5.2904e-04 & 0.9818 & 2.4697e-02 & 0.9947\\
 				~& 800 & 2.0742e-05 & 0.9845 & 6.1490e-05 & 0.9833 & 5.5235e-05 & 0.9793 & 2.6500e-04 & 0.9974\\
 				~ & ~ & ~ & ~ & ~ & ~ & ~ & ~ & ~ & ~\\
 				(2,1)& 50 & 1.2434e-05 & --& 4.0130e-05 & -- & 9.9884e-05 & -- & 3.1689e-04 & --\\
 				~& 100 & 1.4381e-06 & 3.1120 & 7.6089e-06 & 2.3989 & 1.3750e-05 & 2.8608 & 8.1331e-05 & 1.9621 \\
 				~& 200 & 1.8626e-07 & 2.9488 & 1.6364e-06 & 2.2172 & 1.8040e-06 & 2.9302 & 2.0954e-05 & 1.9566 \\
 				~& 400 & 2.3725e-08 & 2.9728 & 3.6593e-07 & 2.1609 & 2.3151e-07 & 2.9620 & 5.3290e-06 & 1.9753\\
 				~& 800 & 2.9811e-09 & 2.9925 & 8.6259e-08 & 2.0848 & 2.9330e-08 & 2.9806 & 1.3438e-06 & 1.9875 \\
 				~ & ~ & ~ & ~ & ~ & ~ & ~ & ~ & ~ & ~\\
 				(3,2)& 50 & 1.8263e-06 & --& 3.9991e-06 & -- & 9.5320e-06 & -- & 2.5282e-05 & --\\
 				~& 100 & 2.0906e-08 & 6.4489 & 2.7533e-07 & 3.8604 & 5.1618e-07 & 4.2068 & 3.2109e-06 & 2.9771 \\
 				~& 200 & 2.2288e-09 & 3.2296 & 3.0566e-08 & 3.1712 & 3.3904e-08 & 3.9283 & 4.2032e-07 & 2.9334 \\
 				~& 400 & 1.3514e-10 & 4.0437 & 3.4272e-09 & 3.1568 & 2.1128e-09 & 4.0042 & 5.3332e-08 & 2.9784 \\
 				~& 800 & 8.4049e-12 & 4.0071 & 4.0911e-10 & 3.0665 & 1.3145e-10 & 4.0065 & 6.7175e-09 & 2.9890\\
 				~ & ~ & ~ & ~ & ~ & ~ & ~ & ~ & ~ & ~\\
 				(4,3)& 50 & 8.8979e-08 & --& 3.5458e-07 & -- & 5.7281e-07 & -- & 2.0679e-06 & --\\
 				~& 100 & 1.2674e-09 & 6.1335 & 1.2599e-08 & 4.8147 & 1.2170e-08 & 5.5567 & 1.2732e-07 & 4.0216 \\
 				~& 200 & 2.6840e-11 & 5.5614 & 6.5149e-10 & 4.2735 & 3.9837e-10 & 4.9331 & 8.2310e-09 & 3.9513\\
 				~& 400 & 6.8886e-13 & 5.2840 & 3.2481e-11 & 4.3260 & 1.3212e-11 & 4.9142 & 5.3332e-10 & 3.9480\\
 				~& 800 & 2.7300e-14 & 4.6572 & 2.0088e-12 & 4.0152 & 4.2361e-13 & 4.9629 & 3.4113e-11 & 3.9666\\     
 				\hline
 			\end{tabular}
 		}
 	\end{center}
    \vspace{-0.4cm}
 	\caption{\scriptsize{$L^2$ errors and convergence rates for the real and imaginary parts of $u$ and $v$ at the final time $T = 2$ for the problem ~\eqref{eq:54} using the S.-flux. The approximations $u^h$ and $v^h$ belong to different approximation spaces with $s = q-1$. The number of cells is denoted by $N$, with a mesh size of $100/N$.}}
    \label{1Dnonlinear_sommer_uv}
 \end{table}

\subsection{Two-dimensional case}

In this section, we investigate the convergence rate of the proposed EDG scheme ~\eqref{eq:21}--\eqref{eq:22} in two-dimensional space. In particular, we consider the following problem
\begin{equation}\label{eq:55}
    u_{tt} - (u_{xx} + u_{yy}) + i (1+e) u_t + e^{|u|^2} u = 0, \quad (x,y) \in (0, 2\pi) \times (0, 2\pi),
\end{equation}
with periodic boundary conditions and initial values $u(x,0) = e^{i(x+y)}, u_t (x,0) = i e^{i(x+y)}\,.$ The exact solution to the problem is $u(x,t) = e^{i(x+y+t)}\,.$

The discretization is performed with elements over the Cartesian grids defined by $(x_k,y_j) = (kh,jh)$, $k,j = 0,1,..., N$ with $h = 2\pi/N$. As in 1D examples, we use three different numerical fluxes: the S.-flux~\eqref{flux:up} with $\xi = 1$, the A.-flux ~\eqref{flux:a1} and the C.-flux~\eqref{flux:central}, but only consider the case where $u_h$ and $v_h$ are in the same approximation space, i.e, $q_x = s_x = q$ and $q_y = s_y = q$.

In Figure~\ref{fig:example5}, we summarize the $L^2$ errors for the real part of $u$ and $v$ against the mesh size $h$ with different degrees of approximation $q$; Table \ref{table_example_2d} presents the corresponding linear least square estimates of the rates of convergence from the error curves in Figure \ref{fig:example5}. For $u$, we observe the optimal convergence rate of $q + 1$ for the S.-flux when $q \geq 2$ and an order reduction of $1$ compared to the optimal convergence rate for $q = 1$. For the A-flux, we observe optimal convergence for $q \geq 3$, an order reduction of $1$ for $q = 2$, of $1.5$ for $q = 1$. For the C-flux, we observe optimal convergence for $q = 1,3,4$ and an order reduction of $1$ for $q = 2$. For $v$, we have an order reduction of $2$ compared to the optimal convergence rate when using the A-flux and of $1$ when using the S-flux; when using the C-flux, there is an order reduction of $1$ for odd $q$ and of $2$ for even $q$ compared to the optimal convergence rate. These observations are consistent with the 1D results.

\begin{figure}[!htb]
\centering
{\includegraphics[width=0.32\textwidth]{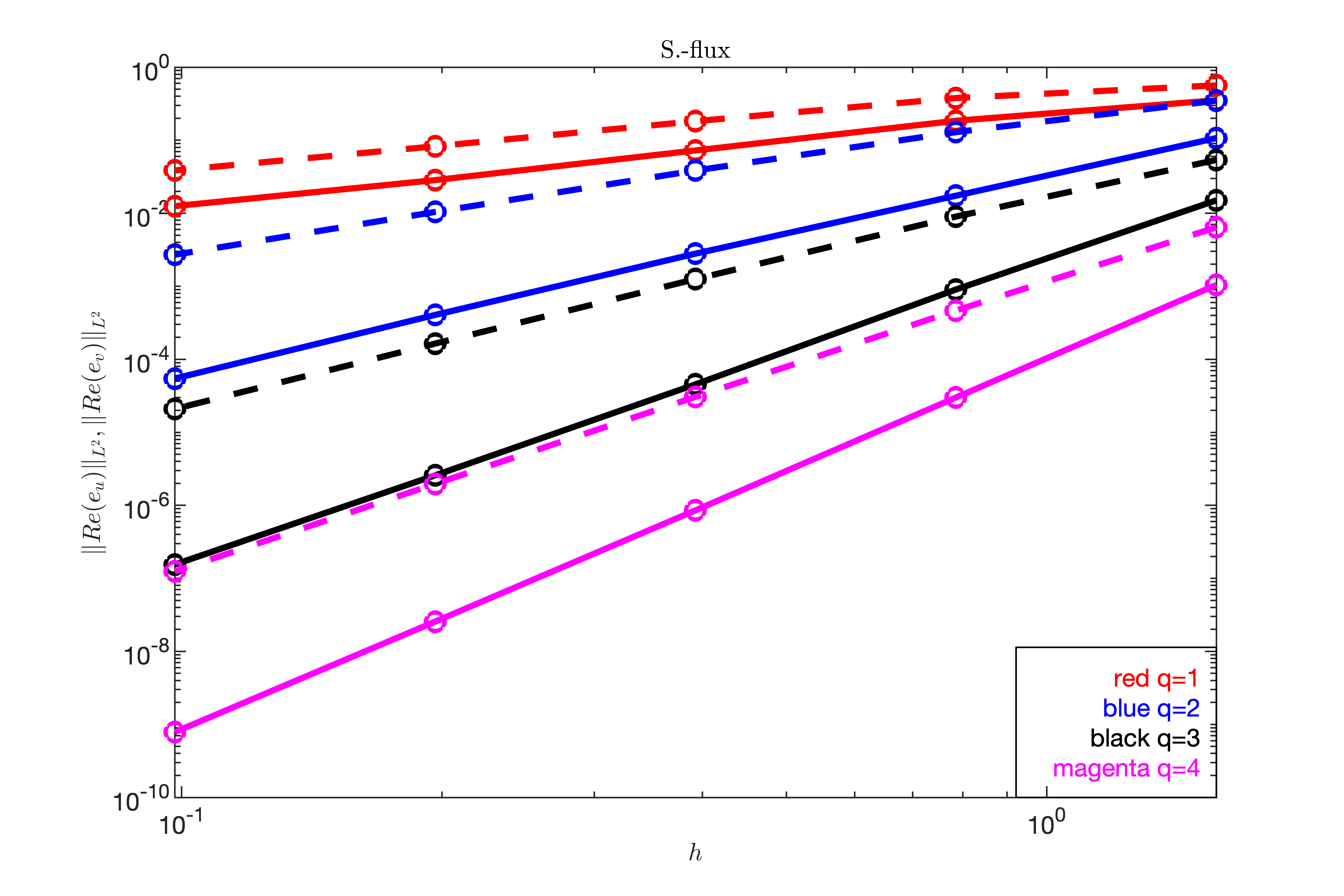}}
{\includegraphics[width=0.32\textwidth]{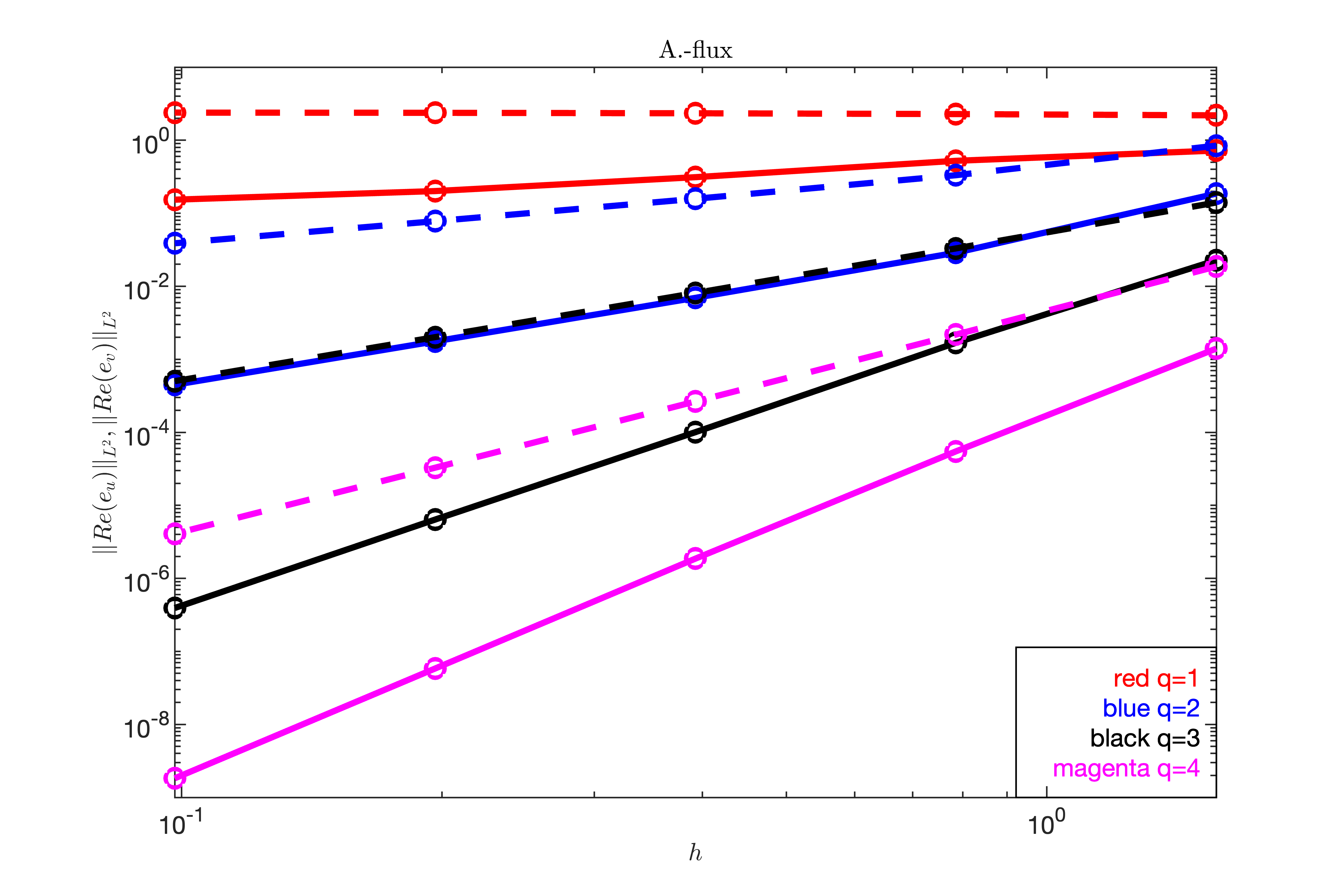}} 
{\includegraphics[width=0.32\textwidth]{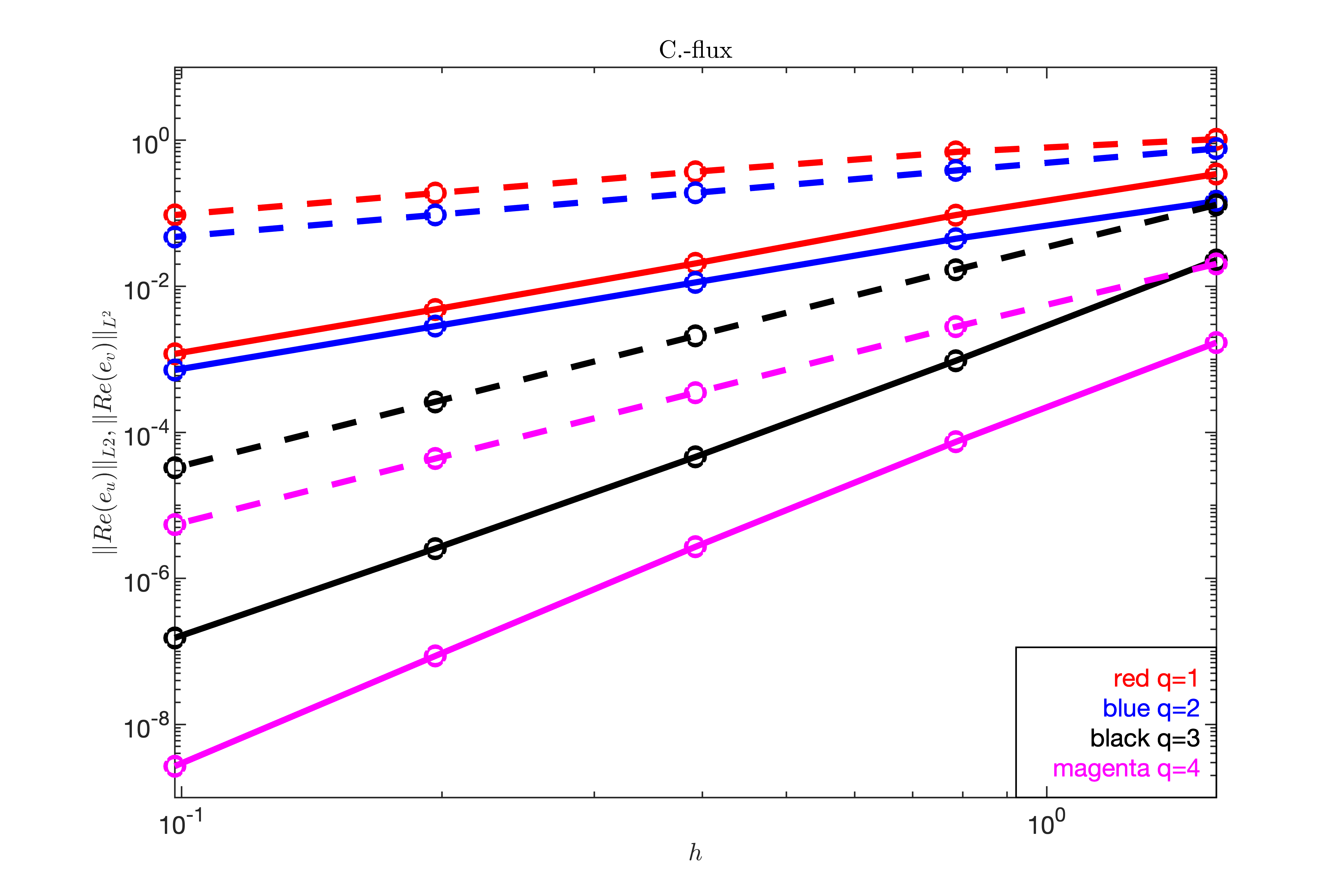}}
    \caption{\scriptsize{$L^2$ errors for the real parts of $u$ and $v$ for the problem ~\eqref{eq:55}. The solid lines show the results of $u_h$, while the dashed lines show the results of $v_h$. $u^h$ and $v^h$ are chosen to be in the same approximation space.}}
    \label{fig:example5}
\end{figure}

\begin{table}[!htb]
 	\footnotesize
 	\begin{center}
 		\scalebox{1.0}{
 			\begin{tabular}{c c c c c c}
 				\hline
 				~ & ~ &  ~& convergence rates of $u/v$  & ~ &~\\
 				\hline
 				flux/$(q,s)$ & ~ & $q=1$ & $q=2$ & $q=3$ & $q=4$ \\
 				\hline
 				S.-flux& ~ & 1.2332/0.9938 & 2.7304/1.7647 & 4.1591/2.8439 & \quad \quad 5.0882/3.9196 \\
                    A.-flux& ~ & 0.5847/-0.0306 & 2.1428/1.0975 & 3.9683/2.0289 & \quad\quad   4.8977/3.0363 \\
 				C.-flux& ~ & 2.0651/0.8774 & 1.9333/1.0062 & 4.2927/2.9930 & \quad \quad 4.8280/2.9716 \\  
 				\hline
 			\end{tabular}
 		}
 	\end{center}
  \vspace{-0.4cm}
 	\caption{\scriptsize{Linear least squares estimates of the rates of convergence from the curves in Figure \ref{fig:example5} for the real parts of $u$ and $v$ of the problem ~\eqref{eq:55}.}}\label{table_example_2d}
 \end{table}

\section{Concluding remarks}
\label{sec:conclusion}

We proposed and analyzed a novel energy-based discontinuous Galerkin scheme for solving nonlinear Schr\"{o}dinger equations with the wave operator. The scheme incorporates a first-order time derivative as an auxiliary variable, which reduces the second-order equation in time to a first-order system in time. By doing so, we only need to introduce one auxiliary regardless of the dimension of the problem, which improves computational efficiency by reducing the memory requirements for the variables to be solved. Moreover, the scheme is unconditionally stable without the need for penalty terms. We demonstrated the stability of the scheme for general mesh-independent numerical fluxes. We also derive suboptimal estimates of convergence in the energy norm and observe, for polynomial degrees above $1$, optimal convergence in the $L^2$ norm for the energy-conserving alternating flux as well as for dissipative methods based on the Sommerfeld flux.

Designing more general numerical fluxes than we proposed here and deriving error estimates for the corresponding schemes is an interesting direction for future study. It is also of great interest to see if the schemes we developed here can be generalized to handle stochastic versions of model~\eqref{EQ:NLSW}. Utilizing the scheme developed here to study inverse problems of reconstructing model parameters from measured solution data of equation~\eqref{EQ:NLSW} is an ongoing project.

\section*{Acknowledgments}

This work is partially supported by the National Science Foundation through grants DMS-1913309, DMS-1937254, and DMS-2309802.

{\small
\bibliography{BIB-REN,BIB-Local}
\bibliographystyle{siam}
}

\end{document}